\documentclass[a4paper]{article}
\usepackage{macro}
\usepackage[left=1.in, right=1.in, bottom=1.5in, top=1.5in]{geometry}
\usepackage{algorithm}
\usepackage{algpseudocode}
\usepackage[authoryear,round]{natbib}
\setcitestyle{authoryear,round,citesep={;},aysep={,},yysep={;}}
\let\cite\citep
\providecommand{\algorithmicfunction}{}
\providecommand{\algorithmicendfunction}{}
\providecommand{\algorithmicinput}{}
\providecommand{\algorithmicoutput}{}
\providecommand{\algorithmicreturn}{}
\AtBeginDocument{%
  \renewcommand{\algorithmicfunction}{\textbf{procedure}}%
  \renewcommand{\algorithmicendfunction}{\algorithmicend\ \textbf{procedure}}%
  \renewcommand{\algorithmicreturn}{\textbf{return}}%
  \newcommand{\STATE}{\State}%
  \newcommand{\FOR}{\For}%
  \newcommand{\ENDFOR}{\EndFor}%
  \newcommand{\IF}{\If}%
  \newcommand{\ENDIF}{\EndIf}%
  \newcommand{\FUNCTION}[1]{\State \algorithmicfunction\ #1}%
  \newcommand{\ENDFUNCTION}{\State \algorithmicendfunction}%
  \newcommand{\RETURN}{\State \algorithmicreturn\ }%
  \renewcommand{\algorithmicinput}{\textbf{Input:}}%
  \newcommand{\Input}{\State \algorithmicinput\ }%
  \renewcommand{\algorithmicoutput}{\textbf{Output:}}%
  \newcommand{\Output}{\State \algorithmicoutput\ }%
}

\title{Adaptive Conditional Gradient Sliding:\\ Projection-Free and Line-Search-Free Acceleration}
\author{Shota Takahashi\thanks{
Graduate School of Information Science and Technology, The University of Tokyo, Tokyo, Japan (\href{mailto:shota@mist.i.u-tokyo.ac.jp}{\texttt{shota@mist.i.u-tokyo.ac.jp}})}}
\date{}
\hypersetup{%
  pdftitle={Adaptive Conditional Gradient Sliding: Projection-Free and Line-Search-Free Acceleration},
  pdfauthor={Shota Takahashi}
}

\begin{document}

\maketitle

\begin{abstract}
  We study convex optimization problems over a compact convex set where projections are expensive but a linear minimization oracle (LMO) is available.
  We propose the \emph{adaptive conditional gradient sliding method} (\adcgs), a projection-free and line-search-free method that retains Nesterov's acceleration with adaptive stepsizes based on local Lipschitz estimates.
  An optional initialization uses a line search only at the first iteration.
  \adcgs combines an accelerated outer scheme with an LMO-based inner routine.
  It reuses gradients across multiple LMO calls to reduce gradient evaluations, while controlling the subproblem inexactness via a prescribed accuracy level coupled with adaptive stepsizes.
  For convex objective functions, we prove an $\O(\hat L_k/k^2)$ bound, yielding accelerated $\O(1/k^2)$ rates when the running local curvature bound $\hat L_k$ is bounded, in particular under local smoothness.
  These rates match projection-based methods without relying on a projection oracle.
  For locally smooth and locally strongly convex objective functions, we further establish linear convergence without additional geometric assumptions on the constraint set, such as polytopes or strongly convex sets.
  Experiments on constrained $\ell_p$ regression, logistic regression, and least-squares problems demonstrate that \adcgs improves over projection-free baselines and provides competitive performance when projections are inexpensive.
\end{abstract}

\section{Introduction}
In this paper, we consider constrained convex optimization problems of the form
\begin{align}
  \min_{x \in P} \quad f(x), \label{prob}
\end{align}
where $f:\R^n \to \R$ is a continuously differentiable convex function and $P \subset \R^n$ is a compact convex set. Let $x^*\in P$ denote an optimal solution.
We assume access to a first-order oracle (FOO) for $f$, \ie, given $x \in \R^n$, we can compute $\nabla f(x)$, and a linear minimization oracle (LMO) for $P$, \ie, given $c \in \R^n$, we can compute $v \in \argmin_{u \in P} \langle c, u\rangle$.
In large-scale constrained optimization, projection-free methods are particularly useful when projections are computationally expensive or intractable, while LMO calls remain cheap for many structured sets~\citep{Combettes2021-ct,Braun2025-sz}.
The \emph{conditional gradient method} (CG), also known as the \emph{Frank--Wolfe method} (FW), is widely used for such projection-free settings.
However, achieving fast convergence in practice with simple, parameter-free updates remains challenging.

\paragraph{Related work:}
The CG method was originally proposed by~\citet{Frank1956-cq} and independently rediscovered and further developed by~\citet{Levitin1966-sm}.
Early complexity lower bounds for the CG method were established by~\citet{Canon1968-sr} and refined by~\citet{GueLat1986-bg}.
More recently, \citet{Jaggi2013-oz} established a refined convergence analysis of CG methods, including a lower bound that highlights a trade-off between sparsity and error.
Concurrently,~\citet{Lan2013-dv} examined the complexity of linear-optimization-based methods and established a comparable lower bound.
Later,~\citet{Lacoste-Julien2015-gb} provided a detailed convergence analysis for CG variants.

On the algorithmic side, \citet{Lan2016-op} proposed the \emph{conditional gradient sliding method} (CGS), a projection-free framework that couples an outer accelerated scheme with inner CG loops.
A key feature of CGS is that it reduces the number of FOO calls by reusing gradients across multiple LMO calls and thereby skipping some gradient evaluations.
A recent improvement of CGS was proposed by~\citet{Garber2025-al}.
A line search strategy for CG was introduced by~\citet{Pedregosa2020-ay}, refined for numerical stability by~\citet{Pokutta2024-qu}, and extended beyond Lipschitz gradients by~\citet{Takahashi2026-py}.
Restart strategies and error-bound analyses have been developed to improve convergence~\citep{Kerdreux2019-ec,Kerdreux2021-fv,Kerdreux2022-tv}.
See~\citep{Pokutta2024-qu,Braun2025-sz} for background on CG methods.

In first-order methods, adaptive stepsizes based on local curvature or Lipschitz estimates have been studied to reduce parameter tuning.
More recently, line-search-free methods, \eg,~\citep{Malitsky2019-dp,Malitsky2020-vs,Malitsky2024-ct,Latafat2025-iu}, have emerged as appealing alternatives in practice, using adaptive stepsizes based on local Lipschitz estimates and avoiding parameter tuning.
Additionally, CG methods using adaptive stepsizes have been proposed by~\citet{Yagishita2025-sx,Khademi2025-gt}.
The \emph{auto-conditioned fast gradient method} (AC-FGM)~\citep{Li2025-cb} is the first method that is line-search-free, combines acceleration with adaptive stepsizes, and provides theoretical guarantees when projections are available.
Related accelerated schemes for unconstrained problems have also been studied~\citep{Suh2025-eq,Borodich2025-jq}.
Both CGS and AC-FGM are rooted in Nesterov's accelerated gradient framework~\citep{Nesterov1983-xn,Nesterov2018-gj} and are based on~\citep{Auslender2006-cy,Tseng2008-fz,Gasnikov2018-by,Ahn2022-zt}.

\paragraph{Limitations of Existing Methods:}
Despite significant progress, CGS depends on global smoothness parameters, which are often unknown or difficult to estimate in practice.
While AC-FGM is fast and line-search-free, it relies on projections and may be impractical when projections are expensive.
Thus, it remains unclear whether accelerated rates can be achieved using adaptive stepsizes without access to a projection oracle.
These gaps motivate an accelerated, projection-free, and line-search-free method that explicitly handles the interaction between adaptive stepsizes and inexact inner solves.
This raises the following question:
\begin{quote}
  \emph{Can we achieve accelerated rates without projection or line-search?}
\end{quote}

\paragraph{Contributions:}
We answer the above question in the affirmative and propose the \emph{adaptive conditional gradient sliding method} (\adcgs), an \textbf{accelerated}, \textbf{projection-free}, and \textbf{line-search-free} method.
Our main contributions are:
\begin{itemize}
  \item \textbf{Method:} We develop \adcgs, whose inner CG loops use only LMO calls and stop at a prescribed FW-gap tolerance. By employing adaptive stepsizes based on local Lipschitz estimates, \adcgs avoids line search in its main iterations and does not require a global smoothness constant, unlike CGS.
  \item \textbf{Key technical mechanism:} Local Lipschitz estimates allow \adcgs to take larger stepsizes, thereby making progress more sensitive to inner inexactness and typically requiring tighter FW-gap tolerances, hence more LMO calls. Our analysis quantifies this trade-off and provides a termination schedule that preserves acceleration.
  In particular, for the fixed-iteration scheme, the FW-gap tolerance decays at the same asymptotic rate as in CGS, even with adaptive stepsizes.
  \item \textbf{Theory and convergence rates:} We establish an $\O(\hat L_k/k^2)$ bound for convex objective functions and an accelerated $\O(1/k^2)$ rate when $\hat L_k$ is bounded, in particular under local smoothness, matching the corresponding projection-based guarantees~\citep{Li2025-cb} while using only an LMO and no projection oracle.
  We also prove linear convergence for locally smooth and locally strongly convex objective functions, without requiring additional structure on the constraint set, such as polytopes or strongly convex sets.
  \item \textbf{Experiments:} Experiments on constrained $\ell_p$ regression, logistic regression with real-world datasets, and least-squares problems show that \adcgs improves over projection-free baselines and is competitive with projection-based methods.
\end{itemize}

\paragraph{Notation:} 
We use $\|\cdot\|$ and $\langle\cdot,\cdot\rangle$ for the Euclidean norm and inner product, and $\|\cdot\|_p$ for the $\ell_p$ norm ($p\geq1$).
For symmetric $A\in\R^{n\times n}$, $\lambda_{\max}(A)$ is its largest eigenvalue.

Let $f:\R^n\to\R$ be differentiable and convex.
For $x,y\in\R^n$, define $D_f(x,y):=f(x)-f(y)-\langle \nabla f(y),x-y\rangle$.
We say $f$ is $L$-smooth if there exists $L > 0$ such that $\|\nabla f(x)-\nabla f(y)\|\leq L\|x-y\|$ for all $x,y\in\R^n$, which implies $\frac{1}{2L}\|\nabla f(x)-\nabla f(y)\|^2 \leq D_f(x,y) \leq \frac{L}{2}\|x-y\|^2$ due to the convexity of $f$.
We say $f$ is $\mu$-strongly convex if there exists $\mu > 0$ such that $f-\frac{\mu}{2}\|\cdot\|^2$ is convex.
We say $f$ is locally smooth (resp. locally strongly convex) if for every compact set $C\subset\R^n$ there exists $L>0$ (resp. $\mu>0$) such that $f$ is $L$-smooth (resp. $\mu$-strongly convex) on $C$.
We adopt the convention that $\frac{0}{0}=0$ and $\frac{a}{0}=+\infty$ for all $a>0$.

\section{Adaptive Conditional Gradient Sliding for Convex Objective Functions}
In this section, we propose \adcgs for solving~\eqref{prob}.

\paragraph{Projection-free accelerated framework:}
Our proposed method is based on CGS~\citep{Lan2016-op}, a projection-free accelerated framework that couples an accelerated outer scheme with inner CG loops.
A key feature of CGS is that it skips some gradient evaluations, allowing multiple LMO calls to reuse a single gradient and thereby reducing the number of FOO calls.
We enhance this CGS framework by incorporating adaptive stepsizes driven by local Lipschitz estimates $L_k$, as in AC-FGM~\citep{Li2025-cb}.
Concretely, at each outer iteration, \adcgs approximately solves a strongly convex subproblem via inner CG loops:
\begin{equation}
  \label{eq:acfg-subprob}
  \min_{z\in P}\biggl\{\langle \nabla f(x_{k-1}), z\rangle + \frac{1}{2\eta_k}\|z-y_{k-1}\|^2\biggr\}.
\end{equation}
We run CG on~\eqref{eq:acfg-subprob} until its FW gap is at most $\delta_k>0$, \ie, $\max_{v\in P}\langle \nabla f(x_{k-1}) + (z_k - y_{k-1})/\eta_k, z_k - v\rangle \leq \delta_k$, and set $z_k$ to the final iterate.
The exact minimizer is the projection of $y_{k-1}-\eta_k\nabla f(x_{k-1})$ onto $P$, but each inner CG step uses only an LMO to minimize a linear model of this quadratic subproblem.
A single CG step need not meet the prescribed tolerance; repeated inner steps reuse the same outer gradient, without invoking a projection oracle.
We then update $y_k = (1-\beta_k)y_{k-1} + \beta_k z_k \in P$ and $x_k = \frac{\tau_k}{1+\tau_k} x_{k-1} + \frac{1}{1+\tau_k} z_k \in P$ with $\beta_k \in [0,1]$ and $\tau_k\geq0$.

\paragraph{Line-search-free framework:}
Following AC-FGM~\citep{Li2025-cb}, we compute $x_1 \neq x_0$ using the initial stepsize $\eta_1$, and choose $\eta_k$ via local Lipschitz estimates:
\begin{align}
  L_1 &= \frac{\|\nabla f(x_1) - \nabla f(x_0)\|}{\|x_1 - x_0\|}, \label{def:adaptive-step-size-1}\\
  L_k &=
  \left\{
  \begin{array}{ll}
    0, & \text{if } D_f(x_{k-1}, x_k) = 0,\\
    \frac{\|\nabla f(x_k) - \nabla f(x_{k-1})\|^2}{2D_f(x_{k-1}, x_k)}, & \text{if } D_f(x_{k-1}, x_k) > 0,
  \end{array}
  \right. \quad \text{for }k\geq 2.\label{def:adaptive-step-size}
\end{align}
We present our method in~\cref{alg:adcgs}, which allows arbitrary nonnegative sequences $\{\beta_k\}$, $\{\delta_k\}$, and $\{\tau_k\}$ with $\{\eta_k\}$ based on $L_k$, defined in~\cref{corollary:convex-sublinear,corollary:convex-fixed-iteration,corollary:convex-sublinear-2}.

\begin{algorithm}[!t]
  \caption{Adaptive Conditional Gradient Sliding method (\adcgs)}\label{alg:adcgs}
  \begin{algorithmic}[1]
    \Input $x_0 = y_0 = z_0 \in P$; $\{\beta_k\}$, $\{\delta_k\}$, $\{\eta_k\}$, $\{\tau_k\}$.
    \FOR{$k = 1,\ldots,$}
      \STATE $z_k = \cg(\nabla f(x_{k-1}), y_{k-1}, \eta_k, \delta_k)$
      \STATE $y_k = (1-\beta_k)y_{k-1} + \beta_k z_k$
      \STATE $x_k = \frac{\tau_k}{1 + \tau_k}x_{k-1} + \frac{1}{1 + \tau_k}z_k$
      \STATE Compute $L_k$ using~\eqref{def:adaptive-step-size-1} for $k=1$ and~\eqref{def:adaptive-step-size} for $k\geq 2$.
      \STATE Set $\eta_{k+1}$ according to the stepsize rule based on $L_k$.\label{line:stepsize-update}
    \ENDFOR
    \Output $x_k$

    \FUNCTION{$\cg(g, u, \eta, \delta)$}\label{alg:cg}
      \STATE $u_1 = u$
      \FOR{$t = 1,\ldots,$}
        \STATE $v_t \in \argmin_{v \in P}\langle g + (u_t - u)/\eta, v\rangle$
        \IF{$\max_{v\in P}\langle g + (u_t - u)/\eta, u_t - v\rangle \leq \delta$}
          \RETURN $u_t$
        \ENDIF
        \STATE $\gamma_t = \min\left\{1, \frac{\langle g + (u_t - u)/\eta,  u_t - v_t\rangle}{\|v_t - u_t\|^2/\eta}\right\}$,
        $u_{t+1} = (1 - \gamma_t)u_t + \gamma_t v_t$
      \ENDFOR
    \ENDFUNCTION
  \end{algorithmic}
\end{algorithm}

\paragraph{Complexity of inner CG iterations:}
We analyze the complexity of inner CG loops for~\eqref{eq:acfg-subprob} to an accuracy level of $\delta_k > 0$.
The proof is given in Appendix~\ref{app:proof-cg-complexity}. 
\begin{theorem}\label{theorem:cg-complexity}
  Let $D := \max_{x,y\in P}\|x-y\|$ be the diameter of $P$.
  The number of CG iterations to solve~\eqref{eq:acfg-subprob} up to accuracy $\delta_k$ in the $k$-th outer iteration is at most $T_k := \left\lceil\frac{6D^2}{\eta_k\delta_k}\right\rceil$.
\end{theorem}

\section{Convergence Analysis for Convex Objective Functions}\label{section:convergence-analysis-convex}
In this section, we establish an $\O(\hat L_k/k^2)$ bound under~\cref{corollary:convex-sublinear}, which yields $\O(1/k^2)$ when $\hat L_k$ is bounded, in particular under local smoothness.
Building on a one-step inequality, the main new ingredient is inexactness in $z_k$: the FW-gap tolerance $\delta_k$ generates an error term that accumulates over iterations.
We show how to bound this accumulated error while still allowing adaptive stepsizes, and this yields our convergence rates.
We now state the one-step inequality under the following conditions on the parameter sequences $\{\tau_k\}$, $\{\beta_k\}$, and $\{\eta_k\}$: $\tau_1 = 0$,
\begin{align}
  \beta_1 &= 0, \quad \beta_k = \beta > 0, k \geq 2,\label{ineq:beta}\\
  \eta_2 &\leq \min\left\{(1 - \beta)\eta_1, \tfrac{1}{4L_1}\right\},\label{ineq:eta-k2}
\end{align}
and, for $k \geq 3$,
\begin{equation}
  \eta_k \leq \min\left\{2(1-\beta)^2\eta_{k-1}, \tfrac{\tau_{k-1}}{4L_{k-1}}, \tfrac{\tau_{k-2} + 1}{\tau_{k-1}}\eta_{k-1}\right\}, \label{ineq:eta-k3}
\end{equation}
where $L_k$ is the local Lipschitz estimate in~\eqref{def:adaptive-step-size-1} and~\eqref{def:adaptive-step-size}.

\begin{proposition}\label[proposition]{proposition:convex-obj-ineq}
  Suppose that $L_k$ is defined in~\eqref{def:adaptive-step-size-1} and~\eqref{def:adaptive-step-size}, and that $\{\tau_k\}$, $\{\beta_k\}$, and $\{\eta_k\}$ satisfy $\tau_1 = 0$,~\eqref{ineq:beta},~\eqref{ineq:eta-k2}, and~\eqref{ineq:eta-k3}.
  Let $\{x_k\}$, $\{y_k\}$, and $\{z_k\}$ be the sequences generated by \adcgs.
  Then, for any $z \in P$, it holds that
  \begin{align}
    &\sum_{i=1}^k\eta_{i+1}(\tau_i f(x_i) + \langle\nabla f(x_i), x_i - z\rangle - \tau_i f(x_{i-1}))+ \frac{\eta_2}{2\eta_1}(\|z_1 - y_1\|^2 + \|z_2 - z_1\|^2) - \Scal_k\nonumber\\
    &\leq\frac{1}{2\beta}\|y_1 - z\|^2 - \frac{1}{2\beta}\|y_{k+1} - z\|^2 + \sum_{i=2}^{k+1}\D_i,\label{ineq:obj-bound}
  \end{align}
  where $\D_i := \eta_i \langle\nabla f(x_{i-1}) - \nabla f(x_{i-2}), z_{i-1} - z_i \rangle -\tfrac{\eta_i\tau_{i-1}}{2L_{i-1}}\|\nabla f(x_{i-1}) - \nabla f(x_{i-2})\|^2-\tfrac{1}{2}\|z_i - y_{i-1}\|^2$ for any $i \geq 2$ and $\Scal_k:=\sum_{i=1}^k\eta_{i+1}(\delta_{i+1} + \delta_i)$.
\end{proposition}
\begin{proof}[Sketch of Proof]
  The inner loop yields the FW gap of~\eqref{eq:acfg-subprob} up to $\delta_k$, namely, it satisfies $\langle \eta_k\nabla f(x_{k-1}) + (z_k-y_{k-1}), z-z_k\rangle \geq -\eta_k\delta_k$ for all $z\in P$ and $k\geq 1$.
  Substituting this inequality into the accelerated scheme and telescoping as in~\citep[Proposition~1]{Li2025-cb} yields~\eqref{ineq:obj-bound}.
  A complete proof is given in Appendix~\ref{app:proof-convex-obj-ineq}.
\end{proof}

The right-hand side of~\eqref{ineq:obj-bound} contains $\sum_{i=2}^{k+1}\D_i$, which must be bounded to obtain convergence rates.
Using the same bounding technique as in~\citep[Theorem~1]{Li2025-cb}, we obtain the following theorem.
The proof is given in Appendix~\ref{app:proof-convex-obj-decrease}.

\begin{theorem}\label[theorem]{theorem:convex-obj-decrease}
  Suppose that $L_k$ is defined in~\eqref{def:adaptive-step-size-1} and~\eqref{def:adaptive-step-size}, and that $\{\tau_k\}$, $\{\beta_k\}$, and $\{\eta_k\}$ satisfy $\tau_1 = 0$,~\eqref{ineq:beta},~\eqref{ineq:eta-k2}, and~\eqref{ineq:eta-k3}.
  Then, the sequences $\{x_k\}$, $\{y_k\}$, and $\{z_k\}$ generated by \adcgs are bounded.
  Moreover, it holds that
  \begin{align*}
    f(x_k) - f(x^*) &\leq \frac{1}{(\tau_k + 1)\eta_{k+1}}(\E + \Scal_k) \quad \text{and}\quad f(\overline{x}_k) - f(x^*) \leq \frac{1}{\sum_{i=2}^{k+1}\eta_i}(\E + \Scal_k),
  \end{align*}
  where $\E := \frac{1}{2\beta}\|z_0 - x^*\|^2 + \frac{\eta_2}{2}(5L_1/2 - 1/\eta_1)\|z_1 - z_0\|^2$ and $\overline{x}_k := (\sum_{i=1}^{k-1}((\tau_i + 1)\eta_{i+1} - \tau_{i+1}\eta_{i+2})x_i + (\tau_k + 1)\eta_{k+1}x_k)/(\sum_{i=2}^{k+1}\eta_i)$, with $\Scal_k$ defined in~\cref{proposition:convex-obj-ineq}.
\end{theorem}

Using~\cref{theorem:convex-obj-decrease}, we now establish convergence rates of \adcgs for convex objective functions under adaptive stepsizes and inexact inner solves controlled by $\{\delta_k\}$. Here, $\Scal_k$ collects the error contributions induced by the inner tolerances $\{\delta_k\}$, and a key step in the next result is to upper bound this accumulation explicitly.
Moreover, $\delta_k$ determines the inner computational effort needed to obtain $z_k$: by~\cref{theorem:cg-complexity}, the $k$-th subproblem can be solved with at most $T_k=\lceil 6D^2/(\eta_k\delta_k)\rceil$ iterations of the inner CG routine, \ie, LMO calls.
The corollaries below provide simple schedules for $\{\delta_k\}$ that keep $\Scal_k$ under control while preserving acceleration. The proof is given in Appendix~\ref{app:proof-corollary-convex-sublinear}.

\begin{corollary}\label[corollary]{corollary:convex-sublinear}
  Suppose that $L_k$ is defined in~\eqref{def:adaptive-step-size-1} and~\eqref{def:adaptive-step-size}.
  Let $\tau_1=0$ and $\tau_k=k/2$ for all $k\ge2$, let $\beta_k$ be given by~\eqref{ineq:beta} with $\beta\in(0,1-\sqrt{6}/3]$, and let $\delta_k=\frac{D_0^2}{k^{1+\theta}(k+1)}$, where $D_0:=\|z_0-x^*\|$ and $\theta>0$.
  Let $\{\eta_k\}$ satisfy
  \begin{align*}
    \eta_k &= \begin{cases}
      \min\{(1 - \beta)\eta_1, \frac{1}{4L_1}\}, &\text{if}\ k = 2,\\
      \min\{\eta_2, \frac{1}{4L_2}\}, &\text{if}\ k = 3,\\
      \min\left\{\frac{k}{k-1}\eta_{k-1}, \frac{k-1}{8L_{k-1}}\right\}, &\text{if}\ k \geq 4.
    \end{cases}
  \end{align*}
  Let $\{x_k\}$, $\{y_k\}$, and $\{z_k\}$ be the sequences generated by \adcgs.
  Then, it holds that $\eta_k \geq \frac{k}{12 \hat{L}_{k-1}}$ for all $k \geq 2$, where $\hat{L}_{k} := \max\left\{\frac{1}{4(1-\beta)\eta_1}, \max_{1\leq i \leq k} L_i\right\}$,
  \begin{align*}
    f(x_k) - f(x^*) \leq \frac{12\hat{L}_k}{k(k + 1)}\Rcal_1, \quad \text{and}\quad f(\overline{x}_k) - f(x^*) \leq \frac{1}{\sum_{i=2}^{k+1}\frac{i}{6\hat{L}_{i-1}}}\Rcal_1,
  \end{align*}
  where $\Rcal_1 := 2\E + \frac{(1-\beta)(2+\theta)\eta_1D_0^2}{\theta}$, with $\E$ and $\overline{x}_k$ defined in~\cref{theorem:convex-obj-decrease}.
\end{corollary}
\begin{remark}[On the choice of $\delta_k$]\label{remark:choice-delta}
  The schedule of $\delta_k$ in~\cref{corollary:convex-sublinear} uses $D_0$, which is unknown in practice.
  In the analysis, $D_0$ can be replaced by the diameter $D$ of $P$, and the same argument only requires controlling $S_k$.
  Moreover, we define $\tilde{L}_k := \max_{2\le i\le k} L_i$ if $\max_{2\le i\le k} L_i > 0$ and $\tilde{L}_k := L_1$ otherwise.
  In our implementations, we use $\tilde{\delta}_k = \frac{\tilde{L}_kD^2}{k^{1+\theta}(k+1)}$ or $\delta_{1:k} = \frac{L_{1:k} D^2}{k^{1+\theta}(k+1)}$, where $L_{1:k}$ denotes the most recent nonzero value among $L_1,\ldots,L_{k}$.
  If $\tilde L_k$ is bounded, these choices preserve the convergence bounds up to constant factors.
  Since $\delta_{1:k} \leq \tilde{\delta}_k$, $\delta_{1:k}$ generally requires more LMO calls.
\end{remark}
Formally setting $\theta = 0$ in~\cref{corollary:convex-sublinear}, we obtain $\Scal_k\leq\eta_2 D_0^2(1 + 2\log (k+1))/2$ and thus $\O(\hat L_k\log k/k^2)$, or $\O(\log k/k^2)$ when $\hat L_k$ is bounded.
Larger $\theta$ improves the factor $(2+\theta)/\theta$ in the bound but tightens the inner tolerance, typically increasing the number of LMO calls.
A practical range is $\theta\in(0,1]$; we use $\theta=10^{-3}$ to prioritize LMO efficiency.
When we fix the number of iterations $N$ and use $\delta_k = \frac{D_0^2}{Nk}$, we have the simpler bound. The proof is given in Appendix~\ref{app:proof-corollary-convex-fixed-iteration}.
\begin{corollary}\label[corollary]{corollary:convex-fixed-iteration}
  Suppose that $L_k$ is defined in~\eqref{def:adaptive-step-size-1} and~\eqref{def:adaptive-step-size}, and fix the number of outer iterations $N\ge1$.
  Let $\delta_k=\frac{D_0^2}{Nk}$ and let $\{\beta_k\}$, $\{\tau_k\}$, and $\{\eta_k\}$ follow the same rule as in~\cref{corollary:convex-sublinear}.
  Let $\{x_k\}$, $\{y_k\}$, and $\{z_k\}$ be the sequences generated by \adcgs.
  Then, it holds that
  \begin{align*}
    f(x_N) - f(x^*) &\leq \frac{12\hat{L}_N}{N(N+1)}\left(2\E + 3(1-\beta)\eta_1 D_0^2\right),
  \end{align*}
  with $\E$ defined in~\cref{theorem:convex-obj-decrease}. Here $D_0$ may be any upper bound on $\|z_0-x^*\|$.
\end{corollary}
The stepsize rule in~\cref{corollary:convex-sublinear} is line-search-free. 
While $\eta_1$ is free to choose, overly large (resp., small) values may yield a bound with a
$\|z_1-z_0\|^2$ term in $\E$ (resp., make $((1-\beta)\eta_1)^{-1}$ dominate $\hat{L}_k$), so
$\eta_1$ should be set at a reasonable scale.
We consider two options for initializing $\eta_1$.
\paragraph{(i) Line-search-free initialization:}
Set $\eta_1=\frac{c}{4(1-\beta)L_0}$ for some $c>0$ and $L_0>0$, where $L_0$ is an initial Lipschitz estimate~\citep{Li2025-cb}. For example, take $z_{-1}\in P$ with $z_{-1}\neq z_0$
(a perturbation of $z_0$) and set
\begin{align}
  L_0 := \frac{\|\nabla f(z_{-1}) - \nabla f(z_0)\|}{\|z_{-1} - z_0\|}. \label{def:initial-curvature}
\end{align}
\paragraph{(ii) Line search at the first iteration:}
To avoid an explicit dependence on the initial displacement (\eg, $\|z_1-z_0\|^2$), perform a line search only at the first iteration to find $\eta_1$ satisfying $\eta_1 \leq \frac{2}{5L_1}$.
The procedure is given in~\cref{alg:ls-first-iteration}.
Assume that $f$ is locally smooth. Since $\{x_k\}$ is bounded by~\cref{corollary:convex-sublinear}, $f$ is $L$-smooth on the set of iterates.
For $\gamma\geq8/5$, the line search guarantees $\frac{2}{5\gamma L} \leq \eta_1 \leq \frac{2}{5L_1}$, which implies $\hat{L}_k \leq \frac{5\gamma L}{8(1-\beta)}$.
Consequently, we obtain
\begin{align*}
  f(x_k) - f(x^*) &\leq \frac{15\gamma L\sigma D_0^2}{4k(k+1)} \quad \text{and}\quad f(\overline{x}_k) - f(x^*) \leq \frac{15\gamma L\sigma D_0^2}{4k(k+3)},
\end{align*}
where $\sigma =\frac{2}{\beta(1-\beta)} + \frac{2(2 + \theta)\eta_1}{\theta}$.
\cref{alg:ls-first-iteration} requires at most $\O(\log_\gamma(L/L_0))$ iterations.

\begin{algorithm}[t]
  \caption{Line Search at the First Iteration for \adcgs}\label{alg:ls-first-iteration}
  \begin{algorithmic}[1]
    \Input{$z_0 \in P$, $\beta\in(0, 1 - \frac{\sqrt{6}}{3}]$, $\delta_1 > 0$, and $L > 0$}
    \STATE Choose $0 < L_0 \leq L$ (\eg,~\eqref{def:initial-curvature}), and $\gamma > 1$
    \FOR{$j = 0, 1, 2, \ldots$}
    \STATE $\eta_{1,j} = \frac{1}{4(1-\beta)L_0\gamma^j}$,
    $z_{1,j} = \cg(\nabla f(z_0), z_0, \eta_{1,j}, \delta_1)$,
    $L_{1,j} := \frac{\|\nabla f(z_{1,j}) - \nabla f(z_0)\|}{\|z_{1,j} - z_0\|}$
    \IF{$\eta_{1,j} \leq \frac{2}{5L_{1,j}}$}
      \RETURN $z_{1,j}, \eta_{1,j}, L_{1,j}$
    \ENDIF
    \ENDFOR
  \end{algorithmic}
\end{algorithm}

\citet{Li2025-cb} proposed an alternative choice of $\tau_k$ that allows larger stepsizes. We also derive accelerated rates. The proof is given in Appendix~\ref{app:proof-corollary-convex-sublinear-2}.
\begin{corollary}\label[corollary]{corollary:convex-sublinear-2}
  Suppose that $L_k$ is defined in~\eqref{def:adaptive-step-size-1} and~\eqref{def:adaptive-step-size}.
  Let $\tau_1=0$, $\tau_2=1$, and $\tau_k=\tau_{k-1}+\frac{\alpha}{2}+\frac{2(1-\alpha)\eta_k L_{k-1}}{\tau_{k-1}}$ for all $k\ge3$ with a constant $\alpha\in[0,1]$, let $\beta_k$ be given by~\eqref{ineq:beta} with $\beta\in(0,1-\sqrt{6}/3]$, and let $\delta_k=\frac{D_0^2}{k^{1+\theta}(k+1)}$ for some $\theta>0$.
  Let $\{\eta_k\}$ satisfy
  \begin{align*}
    \eta_k &= \begin{cases}
      \min\left\{(1 - \beta)\eta_1, \frac{1}{4L_1}\right\}, &\text{if}\ k = 2,\\
      \min\left\{\frac{4}{3}\eta_{k-1},\frac{\tau_{k-2} + 1}{\tau_{k-1}}\eta_{k-1}, \frac{\tau_{k-1}}{4L_{k-1}}\right\}, &\text{if}\ k \geq 3.
    \end{cases}
  \end{align*}
  Let $\{x_k\}$, $\{y_k\}$, and $\{z_k\}$ be the sequences generated by \adcgs.
  Then, it holds that
  \begin{align*}
    f(x_k) - f(x^*) \leq \frac{12\hat{L}_k\Rcal_2 }{(\alpha k + 4 - 2\alpha)(\alpha k + 3 - 2\alpha)} \quad \text{and} \quad f(\overline{x}_k) - f(x^*) \leq \frac{12\hat{L}_k\Rcal_2 }{6k + \alpha k(k-3)},
  \end{align*}
  where $\Rcal_2 := 2\E + \frac{D_0^2(2 + \theta)}{4\theta \underline{L}_k}$ and $\underline{L}_k := \min_{1 \leq i \leq k}L_i$, with $\E$ and $\overline{x}_k$  as in~\cref{theorem:convex-obj-decrease}, and $\hat{L}_k$ as in~\cref{corollary:convex-sublinear}.
\end{corollary}
When $\alpha=1$, \cref{corollary:convex-sublinear-2} recovers \cref{corollary:convex-sublinear}. For $\alpha\in(0,1)$, it gives accelerated variants with different constants,
whereas $\alpha=0$ is included only as an empirical variant and sometimes performs well in practice (see, \eg, \cref{fig:logistic-regression-gisette}).

\section{Linear Rate for Locally Strongly Convex Objective Functions}
\begin{algorithm}[t]
  \caption{\adcgs for locally strongly convex problems}\label{alg:adcgs-strongly-convex}
  \begin{algorithmic}[1]
    \Input $w_0 \in P$; $\gamma\geq8/5$; an estimate $\phi_0> 0$ satisfying $f(w_0) - f(x^*) \leq \phi_0$
    \FOR{$s = 1,\ldots,$}
      \STATE Let $w_s$ be the last point output by the \adcgs (\cref{alg:adcgs} with~\cref{alg:ls-first-iteration}) when run for
      \begin{align*}
        N = \left\lceil\sqrt{\frac{30\gamma L}{\mu}\left(\frac{1}{\beta(1 - \beta)} + \frac{3}{2}\right)}\right\rceil
      \end{align*}
      iterations with parameters $L$ and $\mu$ given by local smoothness and local strong convexity, $x_0 = w_{s-1}$, $\eta_1 > 0$, $\beta_k$ defined in~\eqref{ineq:beta}, $\delta_k = \frac{(1-\beta)\phi_0}{2^s\mu N\eta_k}$, $\eta_k$ and $\tau_k$ defined in~\cref{corollary:convex-sublinear}, and $L_k$ defined in~\eqref{def:adaptive-step-size-1} and~\eqref{def:adaptive-step-size}. In~\cref{alg:ls-first-iteration}, recompute $\delta_1$ using each trial stepsize $\eta_{1,j}$.
    \ENDFOR
    \Output $w_s$
  \end{algorithmic}
\end{algorithm}
In this section, we establish the linear convergence of \adcgs for locally strongly convex objective functions.
We work under the assumption that $f$ is locally strongly convex and locally smooth.
We present \adcgs for locally strongly convex objective functions in~\cref{alg:adcgs-strongly-convex}, which repeatedly runs \cref{alg:adcgs} for a fixed number of iterations $N$.
We then establish the linear convergence rate of \cref{alg:adcgs-strongly-convex}.
Its proof can be found in Appendix~\ref{app:proof-strongly-convex-linear}.
The following oracle-complexity bounds exclude rejected initialization trials.
\begin{theorem}\label{theorem:strongly-convex-linear}
  Suppose that $f$ is locally strongly convex and locally smooth.
  Let $\{w_s\}$ be the sequence generated by~\cref{alg:adcgs-strongly-convex}.
  Then, $f(w_s) - f(x^*) \leq \frac{\phi_0}{2^s}$ holds for $s \geq 0$.
  Equivalently, to achieve $f(w_s) - f(x^*) \leq \epsilon$, the total number of FOO calls is at most $O(N_0\log_2\lceil\max\{\phi_0/\epsilon,1\}\rceil)$, and the total number of LMO calls is at most $O(L D^2/\epsilon + N_0 \log_2\lceil\max\{\phi_0/\epsilon,1\}\rceil)$, where $N_0 = \sqrt{\frac{\gamma L}{\mu}\bigl(\frac{1}{\beta(1 - \beta)} + 1\bigr)}$.
\end{theorem}
Similarly, by choosing a different value of $N$, one can establish linear convergence when $\tau_k$ and $\eta_k$ are defined as in~\cref{corollary:convex-sublinear-2}.
While linear convergence rates are available under additional geometric assumptions, such as polytopes~\citep{Lacoste-Julien2015-gb}, strongly convex sets~\citep{Garber2015-de}, or uniformly convex sets~\citep{Kerdreux2021-kc}, our result applies to an arbitrary compact convex set $P$ and highlights an oracle-complexity trade-off.
Compared with CG, both CGS and \adcgs substantially reduce the number of FOO calls at the cost of additional LMO calls.

\section{Numerical Experiments}\label{sec:numerical-experiments}
In this section, we present numerical experiments to evaluate the performance of our algorithm.
All experiments were run in Julia~1.12 using FrankWolfe.jl~\citep{Besancon2022-lx}\footnote{\url{https://github.com/ZIB-IOL/FrankWolfe.jl}} on a Mac Studio equipped with an Apple M4 Max processor and 128GB of LPDDR5 memory.
In the plots, \adcgs denotes \adcgs with $\alpha = 1.0$. Labels \adcgs:0.5, \adcgs:0.1, and \adcgs:0.0 correspond to $\alpha = 0.5, 0.1, 0.0$, respectively. We use $\beta = 1 - \sqrt{6}/3$, $\eta_1 = \frac{2}{5L_0}$ with $L_0$ defined in~\eqref{def:initial-curvature}, and $\theta = 10^{-3}$.
The effect of choosing $\theta$ was small in our experiments, possibly because we use $D$ in place of $D_0$, which relaxes the inner tolerance.
Methods with the suffix ``(LS1)'' (\eg, \adcgs:0.5(LS1)) apply the line search (\cref{alg:ls-first-iteration}) at the first iteration.
All other \adcgs runs are fully line-search-free.
We stop when $\max_{v \in P}\langle \nabla f(x_k), x_k - v\rangle \leq 10^{-10}$ or when the maximum number of iterations is reached.
Each figure reports both the primal gap and the FW gap, with all y-axes on a logarithmic scale.
When $f(x^*)$ is unknown, we use the lowest objective value attained by any method as a proxy for $f(x^*)$ when computing the primal gap.
We report results versus the number of iterations (FOO calls), wall-clock time, and the number of LMO calls.
For readability, we truncate the y-axis to omit extremely large values.
Appendix~\ref{app:psd-poisson} presents additional experiments on Poisson regression over positive semidefinite matrices with a trace constraint, and Appendix~\ref{app:lmo-gradient-ratios} reports LMO calls per gradient evaluation.

\subsection{Least Squares Regression over the Simplex}
We first consider the least squares regression problem over the simplex, $\min_{x\in\Delta_n} \frac{1}{2}\|Ax - b\|^2$, where the simplex $\Delta_n := \Set{x\in\R^n_+}{\sum_{i=1}^n x_i \leq 1}$, $A\in\R^{m\times n}$, and $b\in\R^m$.
We generated $A$ with i.i.d.\ entries from the uniform distribution on $[0,1]$, drew a random vector uniformly from $[0,1]^n$ and projected it onto $\Delta_n$ to obtain $x^*$, and defined $b := Ax^*$.
The function $f$ is $L$-smooth with $L = \lambda_{\max}(A^\top A)$.
We compared \adcgs (\cref{alg:adcgs}) with CGS~\citep{Lan2016-op}, CG with line search~\citep{Pedregosa2020-ay}, CG with open-loop stepsizes (using $2/(2 + k)$, \eg,~\citep[Algorithm 4]{Pokutta2024-qu}), the auto-conditioned CG method (AC-CGM)~\citep{Yagishita2025-sx}, and the adaptive CG method (ACG)~\citep{Khademi2025-gt}. 
ACG also employs Lipschitz estimates but still relies on backtracking at every iteration.
\adcgs uses $\tilde{\delta}_k$ as the inner tolerance (see Remark~\ref{remark:choice-delta}).
We set $x_0 = (\frac{1}{n}, \ldots, \frac{1}{n})$ and the maximum number of iterations to 10000.
We consider the cases $(m,n)=(1000,200)$ and $(m,n) = (2500,500)$. We conducted experiments on 20 independently generated random instances (random seeds 1--20). In the figures, the solid line shows the mean over the 20 runs, and the shaded area indicates the mean $\pm$ standard deviation.
The results for $(m,n)=(2500,500)$, including the effect of the first-iteration line search (\cref{alg:ls-first-iteration}), are shown in~\cref{fig:least-squares-simplex-base-500-2500,fig:least-squares-simplex-adcgs-500-2500}.
The remaining results (including those for $(m,n)=(1000,200)$) are provided in Appendix~\ref{app:quadratic_simplex} (see~\cref{fig:least-squares-simplex-base-200-1000,fig:least-squares-simplex-adcgs-200-1000}).
\adcgs:0.5 outperformed the other methods and achieved the best performance, while the other \adcgs variants showed similarly strong performance.
The first-iteration line search yields comparable performance for all variants, except for \adcgs with $\alpha = 1.0$.
\adcgs(LS1) with $\alpha = 1.0$ is omitted because it is unstable. 
\adcgs requires more LMO calls than CGS because of tighter FW-gap tolerances.
\begin{figure*}[!tbp]
  \centering
  \includegraphics[width=0.96\textwidth]{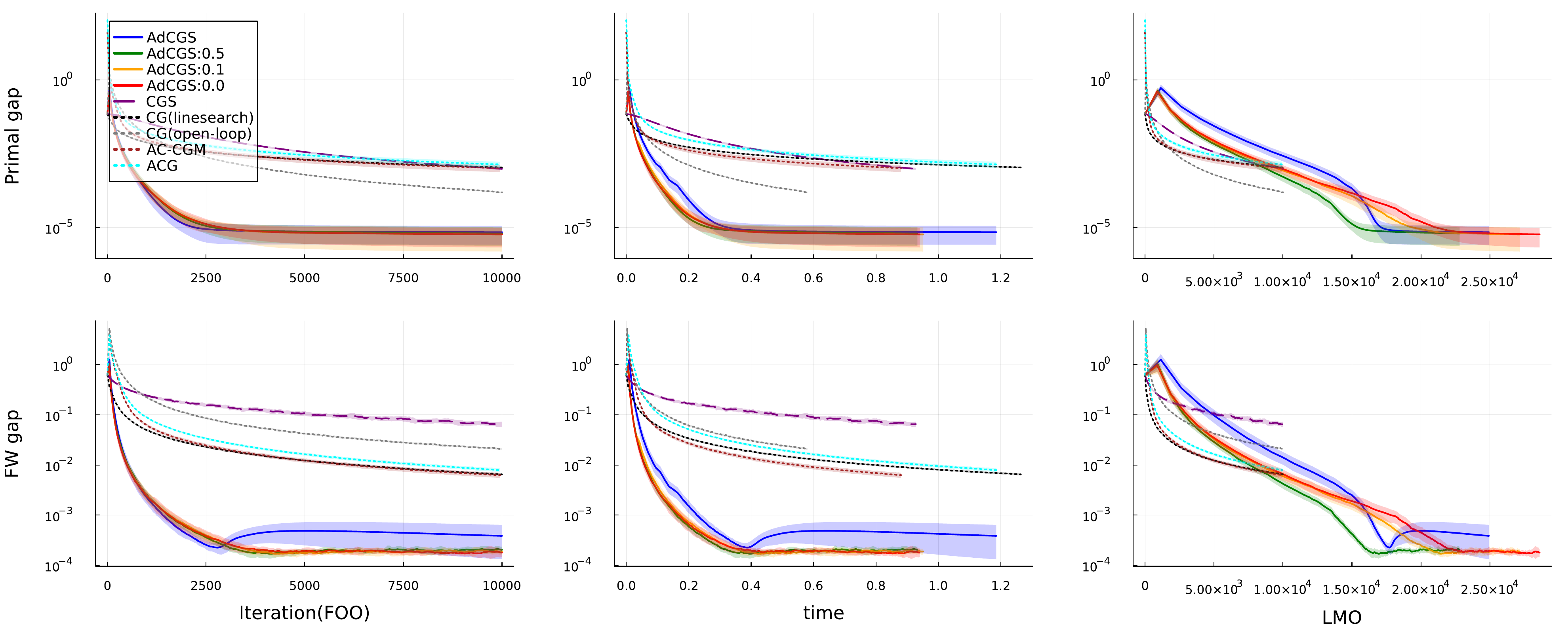}
  \caption{Least squares regression over the simplex with $(m,n) = (2500,500)$ compared with existing methods.}\label{fig:least-squares-simplex-base-500-2500}
\end{figure*}
\begin{figure*}[!tbp]
  \centering
  \includegraphics[width=0.96\textwidth]{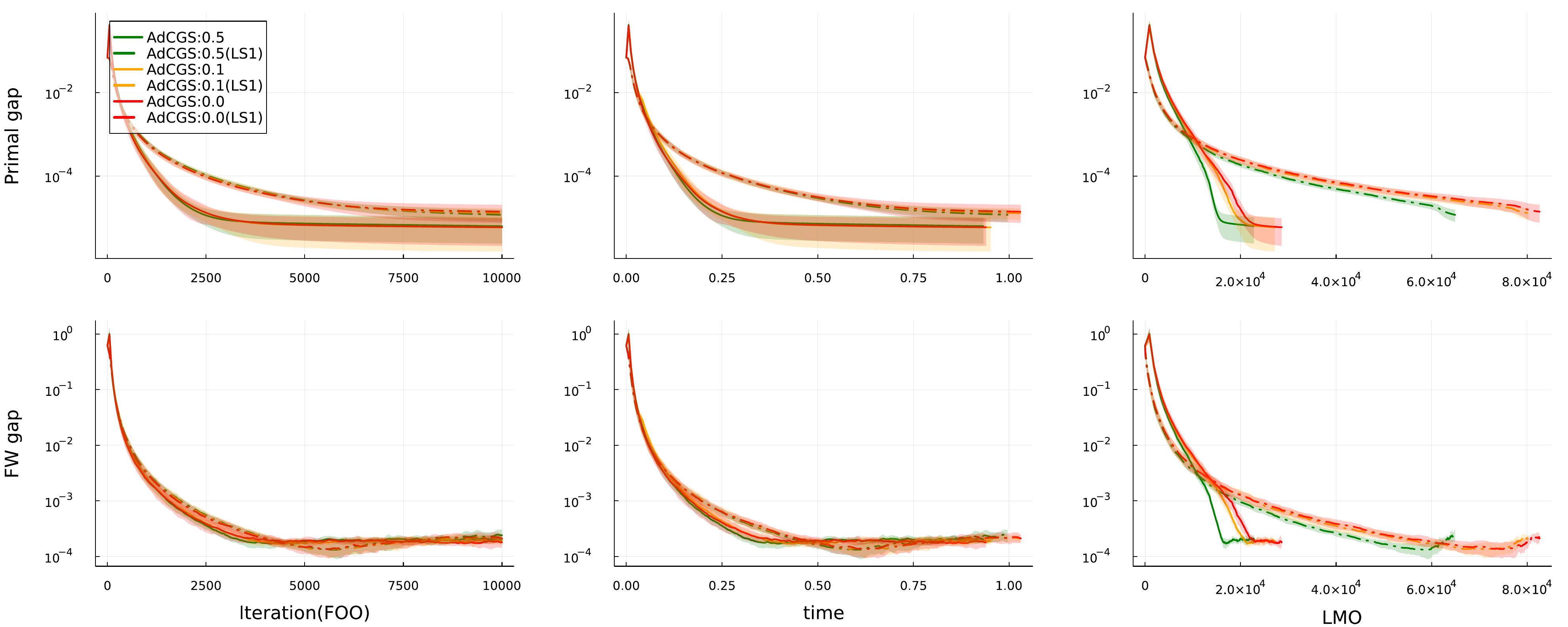}
  \caption{Least squares regression over the simplex with $(m,n) = (2500,500)$ compared with the \adcgs variants.}\label{fig:least-squares-simplex-adcgs-500-2500}
\end{figure*}

\subsection{$\ell_p$ Loss Regression over the $\ell_2$-Ball}
We consider the $\ell_p$ loss regression problem over the $\ell_2$-ball, $\min_{x \in P} \|Ax - b\|_p^p$,
where $A\in\R^{m\times n}$, $b\in\R^m$, $p > 1$, and $P = \Set{x \in \R^n}{\|x\| \leq r}$.
We use the YearPredictionMSD.train ($m = 463715$, $n = 90$) and cpusmall ($m = 8192$, $n = 12$) datasets from LIBSVM~\citep{Chang2011-zh} to form $A$ and $b$.
We standardize the entries of $A$.
We set $p = 1.5$, $r = 0.97\|\tilde{x}\|$ with $\tilde{x} \in \argmin_{x \in \R^n} \|Ax - b\|$, set $x_0$ to the zero vector, and set the maximum number of iterations to 20000.
We compared \adcgs with CGS, AC-CGM, ACG, the projected gradient method (PG), and AC-FGM ($\alpha = 0.5$).
\adcgs uses $\delta_{1:k}$ as the inner tolerance (see Remark~\ref{remark:choice-delta}).
Standard CG methods became ineffective because the CG updates were nearly zero, so we do not report them in this experiment.
We include AC-FGM primarily as a baseline to quantify the performance gap between AC-FGM and \adcgs.
When projections are efficiently computable, AC-FGM typically achieves better performance than \adcgs.
Results for YearPredictionMSD.train are shown in~\cref{fig:year-prediction-msd-pnorm} and results for cpusmall are provided in Appendix~\ref{app:lp_loss} (see~\cref{fig:cpusmall-lp-loss}).
Note that gradient computation is expensive for YearPredictionMSD.train because of its large $m$.
\adcgs:0.0 outperformed existing projection-free methods, PG, and AC-FGM. However, AC-FGM ($\alpha=0.5$) outperformed \adcgs:0.5, likely because AC-FGM computes projections exactly.
For $1 < p < 2$, the objective is not locally smooth, so the theoretical $\O(\hat L_k/k^2)$ bound need not yield $\O(1/k^2)$; we report empirical results using our curvature-scaled tolerance in this regime.
CGS and PG fail to reduce the primal gap in this setting.

\begin{figure*}[!tbp]
  \centering
  \includegraphics[width=0.96\textwidth]{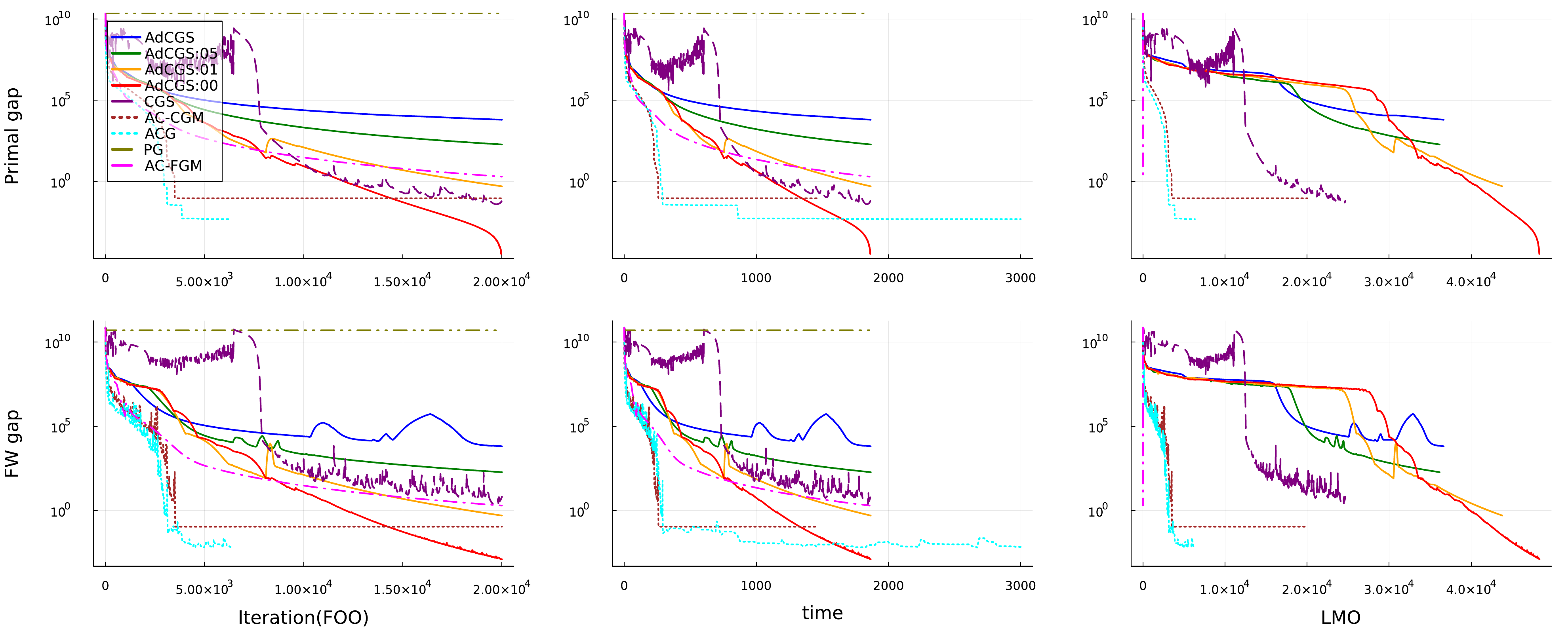}
  \caption{$\ell_p$ loss regression over the YearPredictionMSD.train dataset with $(m,n) = (463715,90)$.}\label{fig:year-prediction-msd-pnorm}
\end{figure*}

\subsection{Logistic Regression}
We consider the logistic regression problem over the $K$-sparse polytope, $\min_{x \in P} \sum_{i=1}^m \log(1 + \exp(-b_i a_i^\top x))$, where $a_i\in\R^n$ and $b_i\in\{-1,1\}$ for $i=1,\ldots,m$ and $P = \Set{x \in \R^n}{\|x\|_1 \leq \kappa K,\|x\|_\infty \leq \kappa}$ with $\kappa > 0$.
We use the a9a ($m = 32561$, $n = 123$, $\kappa = 1$), gisette ($m = 6000$, $n = 5000$, $\kappa = 5$), phishing ($m = 11055$, $n = 68$, $\kappa = 5$), and w8a ($m = 49749$, $n = 300$, $\kappa = 1$) datasets from LIBSVM~\citep{Chang2011-zh}.
We set $K = 0.05n$ and the maximum number of iterations to 20000, and initialized $x_0$ as the zero vector.
We compared \adcgs with CGS, AC-CGM, and ACG.
Since the CG updates were negligible, we omit them.
\adcgs uses $\delta_{1:k}$ as the inner tolerance (see Remark~\ref{remark:choice-delta}).
In this experiment, we set the maximum number of inner iterations for \adcgs to 50.
The results for gisette are shown in~\cref{fig:logistic-regression-gisette}.
Results for the remaining datasets are provided in Appendix~\ref{app:logistic_regression} (see~\cref{fig:logistic-regression-a9a,fig:logistic-regression-w8a,fig:logistic-regression-phishing}).
\adcgs:0.0 outperformed existing methods.

\begin{figure*}[!t]
  \centering
  \includegraphics[width=0.96\textwidth]{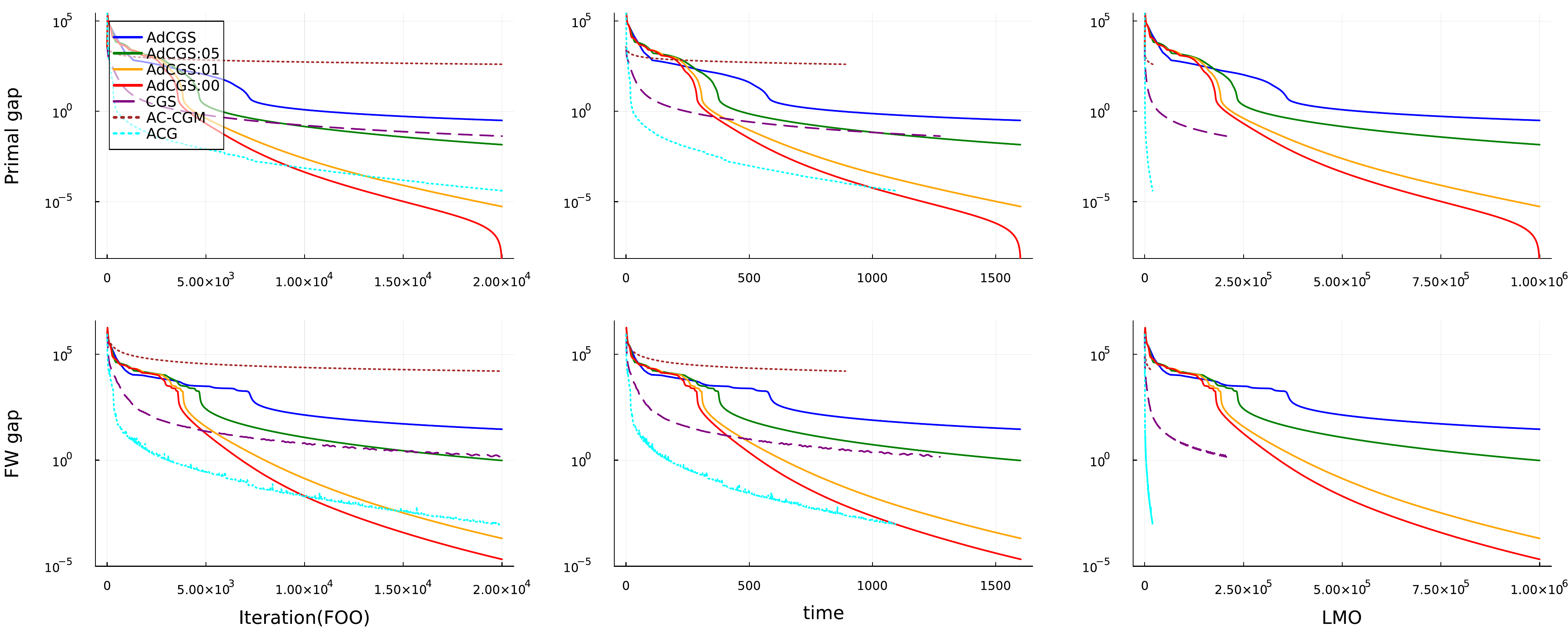}
  \caption{Logistic regression over the gisette dataset with $(m,n) = (6000,5000)$.}\label{fig:logistic-regression-gisette}
\end{figure*}

\section{Conclusion}\label{sec:conclusion}
We proposed \adcgs, an accelerated, projection-free, and line-search-free framework with adaptive stepsizes. 
Our analysis establishes an $\O(\hat L_k/k^2)$ bound for convex objective functions, an $\O(1/k^2)$ rate when $\hat L_k$ is bounded, and linear convergence for locally smooth and locally strongly convex objective functions.
It also clarifies how adaptive stepsizes interact with inexact LMO-based inner solves, yielding oracle-complexity bounds that trade off gradient and LMO calls.
Experiments show that AdCGS improves over projection-free baselines and is competitive with projection-based methods when projections are inexpensive.

\paragraph{Limitations and future work:} 
Linear convergence generally requires local strong convexity. However, algorithms for estimating the associated local strong convexity parameters remain relatively scarce. Although such parameters could plausibly be estimated in a manner analogous to Lipschitz estimates, a rigorous convergence analysis of these estimation procedures has not yet been established for \adcgs.
Developing such estimation procedures remains an important direction for future work. The present analysis covers only two classes of functions: convex functions and locally strongly convex functions.
Future work will extend the framework to nonconvex settings, including weakly convex objective functions.
Other directions include approximate LMOs, stochastic gradients, and inexact gradient evaluations.

\section*{Acknowledgements}
The author is grateful to Akiko Takeda for helpful discussions and constructive feedback.
This project has been funded by the Japan Society for the Promotion of Science (JSPS) and the Nakajima foundation; JSPS KAKENHI Grant Number JP25K21156 and JP23K26336.
The author declares no competing interests.

\bibliography{main}
\bibliographystyle{abbrvnat}

\appendix
\section{Proofs of Theoretical Results}\label{app:proof-of-theoretical-results}
\subsection{Proof of~\cref{theorem:cg-complexity}}\label{app:proof-cg-complexity}
The following lemma will be useful in the analysis.
\begin{lemma}[{\citep[Lemma 2.1]{Lan2016-op}}]
  Let $\lambda_t \in (0,1]$, $t = 1,2,\ldots$ and let 
  \begin{align*}
    \Lambda_t = \begin{cases}
      1, &\text{if } t = 1,\\
      (1 - \lambda_t)\Lambda_{t-1}, &\text{if } t \geq 2.
    \end{cases}
  \end{align*}
  Suppose that $\Lambda_t > 0$ for all $t \geq 2$ and the sequence $\{\zeta_t\}$ satisfies, for any $t \geq 1$ and a sequence $\{\nu_t\}$,
  \begin{align*}
    \zeta_t \leq (1 - \lambda_t)\zeta_{t-1} + \nu_t.
  \end{align*}
  Then, for any $j \in \{1, \ldots, k\}$, it holds that
  \begin{align}
    \zeta_k \leq \Lambda_k\left(\frac{1-\lambda_j}{\Lambda_j}\zeta_{j-1} + \sum_{i=j}^k\frac{\nu_i}{\Lambda_i}\right). \label{ineq:lambda-bound}
  \end{align}
\end{lemma}
The proof follows the same argument as that of~\citep[Theorem 2.2(c)]{Lan2016-op}. We include it here for completeness and to make the presentation self-contained.
Recall the definition of the actual update $u_{t+1}:=(1-\gamma_t)u_t+\gamma_t v_t$, where $\gamma_t = \min\left\{1, \frac{\langle \nabla f(x_{k-1})+ (u_t - u)/\eta_k, u_t - v_t\rangle}{\|v_t - u_t\|^2/\eta_k}\right\}$ is the minimizer of 
\begin{align*}
  \min_{\gamma\in[0,1]}\quad \langle \nabla f(x_{k-1}), (1-\gamma)u_t+\gamma v_t\rangle + \frac{1}{2\eta_k}\|(1-\gamma)u_t+\gamma v_t-y_{k-1}\|^2.
\end{align*}
See the \cg{} subroutine in~\cref{alg:adcgs} for details.
\begin{proof}[Proof of~\cref{theorem:cg-complexity}]
  Let us define $\lambda_t := \frac{2}{t}$ and $\Lambda_t := \frac{2}{t(t-1)}$ for $t \geq 2$. Obviously, $\Lambda_{t+1} := (1 - \lambda_{t+1})\Lambda_t$ holds for all $t \geq 2$. We also define
  \begin{align*}
    \psi_k(z) = \langle \nabla f(x_{k-1}), z\rangle + \frac{1}{2\eta_k}\|z-y_{k-1}\|^2.
  \end{align*}
  Let $\overline{u}_{t+1} = (1 - \lambda_{t+1})u_t + \lambda_{t+1}v_t$. 
  We use $\gamma_t = \min\left\{1, \frac{\langle \nabla f(x_{k-1})+ (u_t - u)/\eta_k, u_t - v_t\rangle}{\|v_t - u_t\|^2/\eta_k}\right\}$, which is the minimizer of $\min_{\gamma\in[0,1]} \psi_k((1-\gamma)u_t+\gamma v_t)$ along the segment between $u_t$ and $v_t$, \ie, $\psi_k(u_{t+1}) \leq \psi_k(\overline{u}_{t+1})$ holds for $u_{t+1}:=(1-\gamma_t)u_t+\gamma_t v_t$.
  Using the fact that $\psi_k(z)$ has a Lipschitz continuous gradient with Lipschitz constant $1/\eta_k$, we have
  \begin{align}
    \psi_k(u_{t+1}) &\leq \psi_k(\overline{u}_{t+1})\nonumber\\
    &\leq \psi_k(u_t) + \langle\nabla \psi_k(u_t), \overline{u}_{t+1} - u_t\rangle + \frac{1}{2\eta_k}\|\overline{u}_{t+1} - u_t\|^2\nonumber\\
    &\eqby{(a)} (1 - \lambda_{t+1})\psi_k(u_t)
    + \lambda_{t+1}(\psi_k(u_t) + \langle\nabla \psi_k(u_t), v_t - u_t\rangle) + \frac{\lambda_{t+1}^2}{2\eta_k}\|v_t - u_t\|^2,\label{ineq:cg-step}
  \end{align}
  where (a) follows from $\overline{u}_{t+1} - u_t = \lambda_{t+1}(v_t - u_t)$.
  By the definition of $v_t$ and the convexity of $\psi_k$, we have $\psi_k(u_t) + \langle\nabla \psi_k(u_t), v_t - u_t\rangle \leq \psi_k(u_t) + \langle\nabla \psi_k(u_t), z - u_t\rangle \leq \psi_k(z)$ for all $z \in P$.
  Therefore, from~\eqref{ineq:cg-step}, we obtain, for all $z \in P$,
  \begin{align}
    \psi_k(u_{t+1}) - \psi_k(z) &\leq (1 - \lambda_{t+1})(\psi_k(u_t) - \psi_k(z)) + \frac{\lambda_{t+1}^2}{2\eta_k}\|v_t - u_t\|^2. \label{ineq:cg-iter}
  \end{align}
  Using~\eqref{ineq:lambda-bound} with indices shifted by one and~\eqref{ineq:cg-iter}, we have, for all $z \in P$,
  \begin{align}
    \psi_k(u_{t+1}) - \psi_k(z) &\leq \Lambda_{t+1}(1 - \lambda_2)(\psi_k(u_1) - \psi_k(z)) + \frac{\Lambda_{t+1}}{\eta_k}\sum_{i=1}^t \frac{\lambda_{i+1}^2}{2\Lambda_{i+1}}\|v_i - u_i\|^2\nonumber\\
    &\leqby{(b)} \frac{2D^2}{\eta_k(t+1)}, \label{ineq:cg-rate}
  \end{align}
  where (b) follows from $\lambda_2 = 1$, $\frac{\lambda_{i+1}^2}{2\Lambda_{i+1}} = \frac{i}{i+1} \leq 1$, and $\|v_i - u_i\| \leq D$ for all $i \geq 1$.
  On the other hand, the FW gap of $\psi_k$ at $u$ is defined as
  \begin{align*}
    \G_k(u) := \max_{v \in P}\langle\nabla \psi_k(u), u - v\rangle.
  \end{align*}
  Using~\eqref{ineq:cg-step}, we have for any $i \geq 1$,
  \begin{align*}
    \lambda_{i+1}\G_k(u_i) &\leq \psi_k(u_i) - \psi_k(u_{i+1}) + \frac{\lambda_{i+1}^2}{2\eta_k}\|v_i - u_i\|^2\leq \xi_i - \xi_{i+1} + \frac{\lambda_{i+1}^2}{2\eta_k}\|v_i - u_i\|^2,
  \end{align*}
  where $\xi_i := \psi_k(u_i) - \psi_k^*$ for all $i \geq 1$ and $\psi_k^* := \min_{z \in P} \psi_k(z)$.
  Dividing both sides by $\Lambda_{i+1}$ and summing from $i=2$ to $t$, we have
  \begin{align*}
    \sum_{i=2}^t\frac{\lambda_{i+1}}{\Lambda_{i+1}}\G_k(u_i) &\leq -\frac{1}{\Lambda_{t+1}}\xi_{t+1} + \sum_{i=3}^t\left(\frac{1}{\Lambda_{i+1}} - \frac{1}{\Lambda_i}\right)\xi_i + 3\xi_2 + \sum_{i=2}^t \frac{\lambda_{i+1}^2}{2\eta_k\Lambda_{i+1}}\|v_i - u_i\|^2\\
    &\leq\sum_{i=3}^t\left(\frac{1}{\Lambda_{i+1}} - \frac{1}{\Lambda_i}\right)\xi_i + 3\xi_2 + \sum_{i=2}^t \frac{\lambda_{i+1}^2}{2\eta_k\Lambda_{i+1}}D^2\\
    &\leqby{(c)} \sum_{i=2}^t i\xi_i + \xi_2 + \frac{t-1}{\eta_k}D^2,
  \end{align*}
  where (c) follows from the definitions of $\lambda_i$ and $\Lambda_i$.
  Using the above inequality, we have, for any $t \geq 2$,
  \begin{align*}
    \min_{1 \leq i \leq t}\G_k(u_i) &\leq \frac{\sum_{i=2}^t\frac{\lambda_{i+1}}{\Lambda_{i+1}}\G_k(u_i)}{\sum_{i=2}^t\frac{\lambda_{i+1}}{\Lambda_{i+1}}}\leq \frac{\sum_{i=2}^t i\xi_i + \xi_2 + \frac{t-1}{\eta_k}D^2}{\sum_{i=2}^t i}\leqby{(d)} \frac{6t-4}{t(t+1)-2}\frac{D^2}{\eta_k}\leq \frac{6D^2}{\eta_k(t+1)},
  \end{align*}
  where (d) follows from~\eqref{ineq:cg-rate}.
  This implies the number of CG iterations required to ensure $\G_k(u_t) \leq \delta_k$ is at most $T_k$.
\end{proof}

\subsection{Proof of~\cref{proposition:convex-obj-ineq}}\label{app:proof-convex-obj-ineq}
\begin{proof}[Proof of~\cref{proposition:convex-obj-ineq}]
Since $\G_k(z_k)\leq \delta_k$, for any $z\in P$ we have
\begin{align}
    \langle \eta_k \nabla f(x_{k-1}) + (z_k-y_{k-1}), z-z_k\rangle \geq -\eta_k\delta_k.
    \label{ineq:opt-cond-z}
\end{align}
Similarly, taking $z\gets z_k$ yields
\begin{align}
    \langle \eta_{k-1}\nabla f(x_{k-2}) + (z_{k-1}-y_{k-2}), z_k-z_{k-1}\rangle \geq -\eta_{k-1}\delta_{k-1}.
    \label{ineq:opt-cond-z-old}
\end{align}
By the update $y_k$, \ie, $y_k = (1-\beta_k)y_{k-1} + \beta_k z_k$, we have $z_{k-1}-y_{k-2}=\frac{1}{1-\beta_{k-1}}(z_{k-1}-y_{k-1})$, and hence~\eqref{ineq:opt-cond-z-old} can be rewritten as
\begin{align}
    \left\langle \eta_k \nabla f(x_{k-2}) + \frac{\eta_k}{\eta_{k-1}(1-\beta_{k-1})}(z_{k-1}-y_{k-1}), z_k-z_{k-1}\right\rangle \geq -\eta_k\delta_{k-1}.
    \label{ineq:opt-cond-z-old-y}
\end{align}
Adding~\eqref{ineq:opt-cond-z} and~\eqref{ineq:opt-cond-z-old-y}, we obtain
\begin{align*}
    &\langle \eta_k \nabla f(x_{k-1}), z-z_{k-1}\rangle
    + \eta_k \langle \nabla f(x_{k-1})-\nabla f(x_{k-2}), z_{k-1}-z_k\rangle\\
    &\quad + \langle z_k-y_{k-1}, z-z_k\rangle
    + \frac{\eta_k}{\eta_{k-1}(1-\beta_{k-1})}\langle z_{k-1}-y_{k-1}, z_k-z_{k-1}\rangle
    \geq -\eta_k(\delta_k+\delta_{k-1}).
\end{align*}
Moreover, applying the identity $2\langle x-y, z-x\rangle = \|y-z\|^2-\|x-y\|^2-\|x-z\|^2$ to both terms in the second line on the left-hand side, we obtain
\begin{align}
  &\langle \eta_k \nabla f(x_{k-1}), z-z_{k-1}\rangle
  + \eta_k \langle \nabla f(x_{k-1})-\nabla f(x_{k-2}), z_{k-1}-z_k\rangle\nonumber\\
  &\quad + \frac{1}{2}\|y_{k-1}-z\|^2 - \frac{1}{2}\|z_k-y_{k-1}\|^2 - \frac{1}{2}\|z_k-z\|^2\nonumber\\
  &\quad + \frac{\eta_k}{2\eta_{k-1}(1-\beta_{k-1})}
  \bigl(\|z_k-y_{k-1}\|^2-\|z_{k-1}-y_{k-1}\|^2-\|z_k-z_{k-1}\|^2\bigr)
  \geq -\eta_k(\delta_k+\delta_{k-1}).
  \label{ineq:opt-cond-z-sum}
\end{align}
In addition, we have
\begin{align}
    \|z_k-z\|^2
    &= \left\|\frac{1}{\beta_k}(y_k-z)-\frac{1-\beta_k}{\beta_k}(y_{k-1}-z)\right\|^2 \nonumber\\
    &\eqby{(a)} \frac{1}{\beta_k}\|y_k-z\|^2 - \frac{1-\beta_k}{\beta_k}\|y_{k-1}-z\|^2
      + \frac{1}{\beta_k}\frac{1-\beta_k}{\beta_k}\|y_k-y_{k-1}\|^2 \nonumber\\
    &\eqby{(b)}
      \frac{1}{\beta_k}\|y_k-z\|^2 - \frac{1-\beta_k}{\beta_k}\|y_{k-1}-z\|^2
      + (1-\beta_k)\|z_k-y_{k-1}\|^2,
    \label{eq:zk-z}
\end{align}
where (a) follows from $\|\alpha a+(1-\alpha)b\|^2 = \alpha\|a\|^2 + (1-\alpha)\|b\|^2 - \alpha(1-\alpha)\|a-b\|^2$ for all $\alpha\in\mathbb{R}$ and (b) follows from the update $y_k = (1-\beta_k)y_{k-1} + \beta_k z_k$.
Substituting~\eqref{eq:zk-z} into~\eqref{ineq:opt-cond-z-sum} and rearranging, we obtain
\begin{align}
    &\langle \eta_k \nabla f(x_{k-1}), z_{k-1}-z\rangle
    - \eta_k(\delta_k+\delta_{k-1})
    + \frac{1}{2\beta_k}\|y_k-z\|^2\nonumber\\
    &\quad+ \frac{\eta_k}{2\eta_{k-1}(1-\beta_{k-1})}\bigl(\|z_{k-1}-y_{k-1}\|^2+\|z_k-z_{k-1}\|^2\bigr) \nonumber\\
    &\quad \leq \frac{1}{2\beta_k}\|y_{k-1}-z\|^2
    + \eta_k\langle \nabla f(x_{k-1})-\nabla f(x_{k-2}), z_{k-1}-z_k\rangle \nonumber\\
    &\quad -\left(\frac{1}{2}+\frac{1-\beta_k}{2}-\frac{\eta_k}{2\eta_{k-1}(1-\beta_{k-1})}\right)\|z_k-y_{k-1}\|^2.
    \label{ineq:opt-cond-z-sum-sub}
\end{align}
For $k\geq 3$, since $\|z_{k-1}-y_{k-1}\|^2+\|z_k-z_{k-1}\|^2 \geq \frac{1}{2}\|z_k-y_{k-1}\|^2$, we can further simplify~\eqref{ineq:opt-cond-z-sum-sub} to obtain
\begin{align}
    &\langle \eta_k \nabla f(x_{k-1}), z_{k-1}-z\rangle
    - \eta_k(\delta_k+\delta_{k-1})
    + \frac{1}{2\beta_k}\|y_k-z\|^2 \nonumber\\
    &\quad \leq \frac{1}{2\beta_k}\|y_{k-1}-z\|^2
    + \eta_k\langle \nabla f(x_{k-1})-\nabla f(x_{k-2}), z_{k-1}-z_k\rangle \nonumber\\
    &\quad -\left(\frac{1}{2}+\frac{1-\beta_k}{2}-\frac{\eta_k}{4\eta_{k-1}(1-\beta_{k-1})}\right)\|z_k-y_{k-1}\|^2 \nonumber\\
    &\quad \leqby{(c)} \frac{1}{2\beta_k}\|y_{k-1}-z\|^2
    + \eta_k\langle \nabla f(x_{k-1})-\nabla f(x_{k-2}), z_{k-1}-z_k\rangle
    - \frac{1}{2}\|z_k-y_{k-1}\|^2,
    \label{ineq:opt-cond-z-bound}
\end{align}
where (c) follows from $\eta_k \leq 2(1-\beta)^2\eta_{k-1}$ in~\eqref{ineq:eta-k3}.
Moreover,
\begin{align}
  &\tau_{k-1}f(x_{k-1}) + \langle \nabla f(x_{k-1}), x_{k-1}-z\rangle\nonumber\\
  &\eqby{(d)} \tau_{k-1}\bigl(f(x_{k-1}) + \langle \nabla f(x_{k-1}), x_{k-2}-x_{k-1}\rangle\bigr)
  + \langle \nabla f(x_{k-1}), z_{k-1}-z\rangle \nonumber\\
  &\eqby{\eqref{def:adaptive-step-size}}
  \tau_{k-1}\left(f(x_{k-2})-\frac{\|\nabla f(x_{k-1})-\nabla f(x_{k-2})\|^2}{2L_{k-1}}\right)
  + \langle \nabla f(x_{k-1}), z_{k-1}-z\rangle,
  \label{eq:from-adaptive-step-size}
\end{align}
where (d) follows from $x_{k-1}-z = z_{k-1}-z + \tau_{k-1}(x_{k-2}-x_{k-1})$ from the update $x_k$.
Combining~\eqref{ineq:opt-cond-z-bound} and~\eqref{eq:from-adaptive-step-size}, we obtain
\begin{align}
  &\eta_k\bigl(\tau_{k-1}f(x_{k-1}) + \langle \nabla f(x_{k-1}), x_{k-1}-z\rangle - \tau_{k-1}f(x_{k-2}) - (\delta_k+\delta_{k-1})\bigr) \nonumber\\
  &\quad \leq \frac{1}{2\beta_k}\|y_{k-1}-z\|^2 - \frac{1}{2\beta_k}\|y_k-z\|^2
  + \eta_k\langle \nabla f(x_{k-1})-\nabla f(x_{k-2}), z_{k-1}-z_k\rangle \nonumber\\
  &\qquad - \frac{1}{2}\|z_k-y_{k-1}\|^2 - \frac{\eta_k\tau_{k-1}}{2L_{k-1}}\|\nabla f(x_{k-1})-\nabla f(x_{k-2})\|^2.
  \label{ineq:opt-cond-adaptive-step-size}
\end{align}
Considering~\eqref{ineq:opt-cond-z-sum-sub} for the case $k=2$ and using $\beta_1=0$, $\tau_1=0$, and $x_1=z_1$, we obtain
\begin{align}
  &\langle \eta_2 \nabla f(x_1), z_1-z\rangle
  - \eta_2(\delta_2+\delta_1)
  + \frac{\eta_2}{2\eta_1}\bigl(\|z_1-y_1\|^2+\|z_2-z_1\|^2\bigr) \nonumber\\
  &\quad \leq \frac{1}{2\beta_2}\|y_1-z\|^2 - \frac{1}{2\beta_2}\|y_2-z\|^2
  + \eta_2\langle \nabla f(x_1)-\nabla f(x_0), z_1-z_2\rangle\nonumber\\
  &\qquad - \left(\frac{1}{2}+\frac{1-\beta_2}{2}-\frac{\eta_2}{2\eta_1}\right)\|z_2-y_1\|^2 \nonumber\\
  &\quad \leqby{(e)} \frac{1}{2\beta_2}\|y_1-z\|^2 - \frac{1}{2\beta_2}\|y_2-z\|^2
  + \eta_2\langle \nabla f(x_1)-\nabla f(x_0), z_1-z_2\rangle
  - \frac{1}{2}\|z_2-y_1\|^2,
  \label{ineq:opt-cond-z-sum-sub-k2}
\end{align}
where (e) holds since $\eta_2\leq (1-\beta)\eta_1$ by~\eqref{ineq:eta-k2}.
Replacing the index by $i$ and summing~\eqref{ineq:opt-cond-z-sum-sub-k2} with~\eqref{ineq:opt-cond-adaptive-step-size} for $i=3,\ldots,k+1$, and using $\tau_1=0$ and $\beta_i=\beta$ for $i\ge2$, we obtain
\begin{align*}
  &\sum_{i=1}^k \eta_{i+1}\bigl(\tau_i f(x_i) + \langle \nabla f(x_i), x_i-z\rangle - \tau_i f(x_{i-1}) - (\delta_{i+1}+\delta_i)\bigr)\\
  &\quad + \frac{\eta_2}{2\eta_1}\bigl(\|z_1-y_1\|^2+\|z_2-z_1\|^2\bigr)\\
  &\quad \leq \frac{1}{2\beta}\|y_1-z\|^2 - \frac{1}{2\beta}\|y_{k+1}-z\|^2 + \sum_{i=2}^{k+1}\D_i,
\end{align*}
where, for any $i \geq 2$, $\D_i := \eta_i \langle\nabla f(x_{i-1}) - \nabla f(x_{i-2}), z_{i-1} - z_i \rangle - \frac{\eta_i\tau_{i-1}}{2L_{i-1}}\|\nabla f(x_{i-1}) - \nabla f(x_{i-2})\|^2-\frac{1}{2}\|z_i - y_{i-1}\|^2$.
We have the desired result by rearranging it.
\end{proof}

\subsection{Proof of~\cref{theorem:convex-obj-decrease}}\label{app:proof-convex-obj-decrease}
The proof follows the same argument as that of~\citep[Theorem~1]{Li2025-cb}. We include it here for completeness and to make the presentation self-contained.
\begin{proof}[Proof of~\cref{theorem:convex-obj-decrease}]
We bound $\D_i$ for $i \geq 3$ and obtain
\begin{align*}
  \D_i &\leqby{(a)} \eta_i\|\nabla f(x_{i-1})-\nabla f(x_{i-2})\|\|z_{i-1}-z_i\|\\
  &\quad - \frac{\eta_i\tau_{i-1}}{2L_{i-1}}\|\nabla f(x_{i-1}) - \nabla f(x_{i-2})\|^2 - \frac{1}{2}\|z_i-y_{i-1}\|^2\\
  &\leqby{(b)}\frac{\eta_i L_{i-1}}{2\tau_{i-1}}\|z_{i-1}-z_i\|^2 - \frac{1}{2}\|z_i-y_{i-1}\|^2\\
  &\leqby{(c)} \frac{\eta_i L_{i-1}(1 - \beta)^2}{\tau_{i-1}}\|z_{i-1}-y_{i-2}\|^2 - \left(\frac{1}{2} - \frac{\eta_i L_{i-1}}{\tau_{i-1}}\right)\|z_i-y_{i-1}\|^2\\
  &\leqby{(d)} \frac{1}{4}\|z_{i-1}-y_{i-2}\|^2 - \frac{1}{4}\|z_i-y_{i-1}\|^2,
\end{align*}
where (a) follows from the Cauchy--Schwarz inequality, (b) follows from Young's inequality, (c) follows from $\|z_{i-1} - z_i\|^2 \leq 2(1 - \beta)^2\|z_{i-1}-y_{i-2}\|^2 + 2\|z_i - y_{i-1}\|^2$, and (d) follows from $\eta_i \leq \frac{\tau_{i-1}}{4L_{i-1}}$ due to~\eqref{ineq:eta-k3}. Moreover, it holds that
\begin{align*}
  \D_2 &\leqby{(e)} \eta_2 L_1\|x_1-x_0\|\|z_1-z_2\|-\frac12\|z_2-y_1\|^2 \\
  &\leqby{(f)} \eta_2 L_1\|z_1-y_0\|(\|z_1-y_0\|+\|z_2-y_0\|)-\frac12\|z_2-y_0\|^2\\
  &\leqby{(g)} \frac{5\eta_2 L_1}{4}\|z_1-z_0\|^2+\left(\eta_2 L_1-\frac12\right)\|z_2-y_0\|^2\\
  &\leqby{(h)} \frac{5\eta_2 L_1}{4}\|z_1-z_0\|^2-\frac14\|z_2-y_1\|^2,
\end{align*}
where (e) follows from the Cauchy--Schwarz inequality and~\eqref{def:adaptive-step-size-1}, (f) follows from the triangle inequality, $x_1 = z_1$, and $x_0 = y_0 = z_0$ due to $\tau_1 = \beta_1 = 0$, (g) follows from Young's inequality, and (h) follows from $\eta_2 \leq \frac{1}{4L_1}$ due to~\eqref{ineq:eta-k2}.
Substituting these inequalities into~\eqref{ineq:obj-bound} yields
\begin{align*}
  &\sum_{i=1}^k \eta_{i+1}\left(\tau_i f(x_i) + \langle \nabla f(x_i), x_i-z\rangle - \tau_i f(x_{i-1}) - (\delta_{i+1}+\delta_i)\right)\\
  &\quad \leq \frac{1}{2\beta}\|y_1-z\|^2 - \frac{1}{2\beta}\|y_{k+1}-z\|^2
  + \left(\frac{5\eta_2L_1}{4}-\frac{\eta_2}{2\eta_1}\right)\|z_1-z_0\|^2 - \frac{1}{4}\|z_{k+1}-y_k\|^2.
\end{align*}
Taking $z=x^*$, using convexity of $f$ (so that $\langle \nabla f(x_i), x_i-x^*\rangle \geq f(x_i)-f(x^*)$), and using the update $x_k$, we obtain
\begin{align}
  &\sum_{i=1}^{k-1}\left((\tau_i+1)\eta_{i+1}-\tau_{i+1}\eta_{i+2}\right)(f(x_i)-f(x^*))
  + (\tau_k+1)\eta_{k+1}(f(x_k)-f(x^*))\nonumber\\
  &\quad \leq \tau_1\eta_2(f(x_0)-f(x^*)) + \frac{1}{2\beta}\|y_1-x^*\|^2
  + \left(\frac{5\eta_2L_1}{4}-\frac{\eta_2}{2\eta_1}\right)\|z_1-z_0\|^2 \nonumber\\
  &\quad\ \ - \frac{1}{2\beta}\|y_{k+1}-x^*\|^2 - \frac{1}{4}\|z_{k+1}-y_k\|^2
  + \sum_{i=1}^k \eta_{i+1}(\delta_{i+1}+\delta_i)\nonumber\\
  &\quad \leqby{(i)} \frac{1}{2\beta}\|z_0-x^*\|^2
  + \left(\frac{5\eta_2L_1}{4}-\frac{\eta_2}{2\eta_1}\right)\|z_1-z_0\|^2
  - \frac{1}{2\beta}\|y_{k+1}-x^*\|^2 - \frac{1}{4}\|z_{k+1}-y_k\|^2\nonumber\\
  &\quad + \sum_{i=1}^{k}\eta_{i+1}(\delta_{i+1}+\delta_i), \label{ineq:obj-bound-final}
\end{align}
where (i) follows from $\tau_1=0$ and $y_1=z_0$.
By~\eqref{ineq:eta-k3}, we have $\eta_i \leq \frac{\tau_{i-2}+1}{\tau_{i-1}}\eta_{i-1}$, equivalently,
$(\tau_{i-2}+1)\eta_{i-1}-\eta_i\tau_{i-1}\geq 0$. Therefore, letting $\Scal_k:=\sum_{i=1}^k\eta_{i+1}(\delta_{i+1} + \delta_i)$, we have
\begin{align*}
  f(x_k)-f(x^*)
  \leq \frac{1}{(\tau_k+1)\eta_{k+1}}
  \left(
    \frac{1}{2\beta}\|z_0-x^*\|^2 + \left(\frac{5\eta_2L_1}{4}-\frac{\eta_2}{2\eta_1}\right)\|z_1-z_0\|^2 + \Scal_k
  \right).
\end{align*}

Furthermore, define
\begin{align*}
  \overline{x}_k = \frac{\sum_{i=1}^{k-1}\bigl((\tau_i+1)\eta_{i+1}-\tau_{i+1}\eta_{i+2}\bigr)x_i + (\tau_k+1)\eta_{k+1}x_k}{\sum_{i=2}^{k+1}\eta_i}.
\end{align*}
Since $\sum_{i=1}^{k-1}\bigl((\tau_i+1)\eta_{i+1}-\tau_{i+1}\eta_{i+2}\bigr) + (\tau_k+1)\eta_{k+1} = \sum_{i=2}^{k+1}\eta_i$, we obtain
\begin{align*}
  f(\overline{x}_k)-f(x^*)
  \leq \frac{1}{\sum_{i=2}^{k+1}\eta_i}
  \left(
    \frac{1}{2\beta}\|z_0-x^*\|^2
    + \left(\frac{5\eta_2L_1}{4}-\frac{\eta_2}{2\eta_1}\right)\|z_1-z_0\|^2
    + \Scal_k
  \right).
\end{align*}
All iterates $x_k$, $y_k$, and $z_k$ belong to the compact set $P$, so they are bounded.
\end{proof}

\subsection{Proof of~\cref{corollary:convex-sublinear}}\label{app:proof-corollary-convex-sublinear}
\begin{proof}[Proof of~\cref{corollary:convex-sublinear}]
  The stepsize $\eta_k$ satisfies~\eqref{ineq:eta-k2} and~\eqref{ineq:eta-k3} from the definition of $\tau_k$ and $\beta$.
  By induction, we show that $\eta_k \geq \frac{k}{12 \hat{L}_{k-1}}$ for all $k \geq 2$ (its proof is given below).
  By substituting this bound into~\cref{theorem:convex-obj-decrease}, we have
  \begin{align}
    f(x_k) - f(x^*) &\leq \frac{12\hat{L}_k}{k(k + 1)}(2\E + 2\Scal_k) \label{ineq:convex-sublinear-bound}
  \end{align}
  and
  \begin{align}
    f(\overline{x}_k) - f(x^*) &\leq \frac{1}{\sum_{i=2}^{k+1}\frac{i}{6\hat{L}_{i-1}}}(2\E + 2\Scal_k). \label{ineq:convex-sublinear-bound-avg}
  \end{align}
  We show $\eta_k \leq \frac{k}{2}\eta_2$ for $k \geq 2$ by induction. From the definition of $\eta_k$, the base cases $k=2$ and $k=3$ hold trivially.
  Suppose that the claim holds for some $k \geq 3$.
  Then, from the definition of $\eta_{k+1}$, we have
  \begin{align*}
    \eta_{k+1} \leq \frac{k+1}{k}\eta_k \leq \frac{k+1}{k}\cdot\frac{k}{2}\eta_2 = \frac{k+1}{2}\eta_2.
  \end{align*}
  Thus, we have $\eta_k \leq \frac{k}{2}\eta_2$ for all $k \geq 2$.
  Therefore, we can bound $\Scal_k=\sum_{i=1}^k\eta_{i+1}(\delta_{i+1} + \delta_i)$ as
  \begin{align*}
    \Scal_k
    &\leqby{(a)} \sum_{i=1}^k \frac{i+1}{2}\eta_2\left(\frac{D_0^2}{(i+1)^{1+\theta}(i+2)} + \frac{D_0^2}{i^{1+\theta}(i+1)}\right) \leq \frac{\eta_2 D_0^2}{2}\sum_{i=1}^k\left(\frac{1}{(i+1)^{1+\theta}} + \frac{1}{i^{1+\theta}}\right)\\
    &\leqby{(b)} \frac{\eta_2 D_0^2}{2}\left(1 + \int_1^{k+1} \frac{2}{x^{1+\theta}}dx\right)= \frac{\eta_2 D_0^2}{2}\left(1 + \frac{2}{\theta}(1 - (k+1)^{-\theta})\right) \leq \frac{\eta_2 D_0^2(2+\theta)}{2\theta},
  \end{align*}
  where (a) follows from $\eta_{i+1} \leq \frac{i+1}{2}\eta_2$ and (b) follows from the integral test.
  By substituting this bound into~\eqref{ineq:convex-sublinear-bound} and~\eqref{ineq:convex-sublinear-bound-avg} and using $\eta_2 \leq (1-\beta)\eta_1$, we obtain the desired result.
\end{proof}
The proof of $\eta_k \geq \frac{k}{12\hat{L}_{k-1}}$ follows the same argument as that of~\citep[Corollary 1]{Li2025-cb}. We include it here for completeness and to make the presentation self-contained.
\begin{proof}[Proof of $\eta_k \geq \frac{k}{12\hat{L}_{k-1}}$]
  We prove~$\eta_k \geq \frac{k}{12\hat{L}_{k-1}}$ by induction on $k$.
  Recall the definition of $\hat{L}_k$ as follows:
  \begin{align*}
    \hat{L}_k := \max\left\{\frac{1}{4(1-\beta)\eta_1},\ \max_{1\leq i\leq k} L_i\right\}.
  \end{align*}
  For $k=2$, by the definition of $\eta_2$ we have
  \begin{align*}
    \eta_2=\min\left\{(1-\beta)\eta_1,\ \frac{1}{4L_1}\right\}=\frac{1}{4\hat{L}_1}.
  \end{align*}
  For $k=3$,
  \begin{align*}
    \eta_3=\min\left\{\eta_2,\ \frac{1}{4L_2}\right\}
    =\min\left\{\frac{1}{4\hat{L}_1},\ \frac{1}{4L_2}\right\} = \frac{1}{4\hat{L}_2}
    \geq\frac{3}{12\hat{L}_2}.
  \end{align*}
  Now assume that for some $k\geq3$, $\eta_k \geq \frac{k}{12\hat{L}_{k-1}}$.
  Then, using the stepsize rule,
  \begin{align*}
    \eta_{k+1}=\min\left\{\frac{k+1}{k}\eta_k,\ \frac{k}{8L_k}\right\}
    \geq \min\left\{\frac{k+1}{12\hat{L}_{k-1}},\ \frac{k}{8L_k}\right\}
    \geq \min\left\{\frac{k+1}{12},\ \frac{k}{8}\right\}\cdot \frac{1}{\hat{L}_k}.
  \end{align*}
  Since $k\geq3$ implies $\frac{k}{8}\geq \frac{k+1}{12}$, the minimum equals $\frac{k+1}{12}$, and hence
  \begin{align*}
    \eta_{k+1}\geq \frac{k+1}{12\hat{L}_k},
  \end{align*}
  which completes the induction.
\end{proof}

\subsection{Proof of~\cref{corollary:convex-fixed-iteration}}\label{app:proof-corollary-convex-fixed-iteration}
\begin{proof}[Proof of~\cref{corollary:convex-fixed-iteration}]
  Using $\eta_k \leq \frac{k}{2}\eta_2$ and $\delta_k = \frac{D_0^2}{Nk}$, we now bound $\Scal_k=\sum_{i=1}^k\eta_{i+1}(\delta_{i+1} + \delta_i)$ as follows: 
  \begin{align*}
    \Scal_N 
    &\leq 
    \frac{\eta_2 D_0^2}{2N}\sum_{i=1}^N\left(1 + \frac{i+1}{i}\right)
    \leqby{(a)}\eta_2 D_0^2\left(1 + \frac{1 + \log N}{2N}\right) \leq \frac{3}{2}\eta_2 D_0^2,
  \end{align*}
  where (a) follows from the integral test.
  Therefore, by substituting this bound into~\eqref{ineq:convex-sublinear-bound} with $k=N$ and using $\eta_2 \leq (1-\beta)\eta_1$, we obtain the desired result.
\end{proof}

\subsection{Proof of~\cref{corollary:convex-sublinear-2}}\label{app:proof-corollary-convex-sublinear-2}
\begin{proof}[Proof of~\cref{corollary:convex-sublinear-2}]
  Using $2(1 - \beta)^2 \geq 4/3$ due to $\beta\in(0,1 - \sqrt{6}/3]$ and the setting of $\eta_k$, we find that the stepsize $\eta_k$ satisfies~\eqref{ineq:eta-k3}.
  By the definition of $\tau_k$, we have for any $k \geq 3$,
  \begin{align}
    \tau_k &= \tau_2 + \sum_{i = 3}^{k}\left(\frac{\alpha}{2} + \frac{2(1-\alpha)\eta_i L_{i-1}}{\tau_{i-1}}\right) \leqby{(a)} 1 + \sum_{i=3}^k\frac{1}{2} = \frac{k}{2}, \label{ineq:tau}
  \end{align}
  where (a) follows from $\tau_2 = 1$ and $\eta_i \leq \frac{\tau_{i-1}}{4L_{i-1}}$. Therefore, for any $k \geq 3$, we also obtain
  \begin{align}
    \frac{\tau_{k-1}+1}{\tau_k} \geqby{(b)} \frac{\tau_k + 1/2}{\tau_k} \geqby{(c)} \frac{k+1}{k}, \label{ineq:tau-ratio}
  \end{align}
  where (b) follows from $\tau_k\leq \tau_{k-1} + 1/2$ due to the definitions of $\eta_k$ and $\tau_k$ and (c) follows from~\eqref{ineq:tau}.
  These inequalities imply that $\tau_k \geq 1 + \frac{\alpha(k-2)}{2}$.
  By induction, we can show $\eta_k \geq \frac{3 + \alpha(k-3)}{12 \hat{L}_{k-1}}$ for all $k \geq 2$ (its proof is given below).
  We bound $\Scal_k$ as follows:
  \begin{align*}
    \Scal_k
    &\leqby{(d)} \sum_{i=1}^k\frac{i+1}{8L_i}(\delta_{i+1} + \delta_i)= \sum_{i=1}^k \frac{i+1}{8L_i}\left(\frac{D_0^2}{(i+1)^{1+\theta}(i+2)} + \frac{D_0^2}{i^{1+\theta}(i+1)}\right)\\
    &\leq \frac{D_0^2}{8\underline{L}_k}\sum_{i=1}^k\left(\frac{1}{(i+1)^{1+\theta}} + \frac{1}{i^{1+\theta}}\right) \leq \frac{D_0^2}{8\underline{L}_k}\left(1 + \frac{2}{\theta}\right),
  \end{align*}
  where (d) follows from $\eta_2 \leq \frac{1}{4L_1}$, $\eta_{i+1} \leq \frac{\tau_i}{4L_i}$ for $i\geq2$, and~\eqref{ineq:tau}.
  Substituting this into~\cref{theorem:convex-obj-decrease}, we obtain the desired results.
\end{proof}
The proof of $\eta_k \geq \frac{3 + \alpha(k-3)}{12\hat{L}_{k-1}}$ follows the same argument as that of~\citep[Corollary 2]{Li2025-cb}. We include it here for completeness and to make the presentation self-contained.
\begin{proof}[Proof of $\eta_k \geq \frac{3 + \alpha(k-3)}{12\hat{L}_{k-1}}$]
  We prove the claim by induction on $k$.
  For $k=2$, by the definition of $\eta_2$, we have
  \begin{align*}
    \eta_2 = \min\left\{(1-\beta)\eta_1,\ \frac{1}{4L_1}\right\} = \frac{1}{4\hat{L}_1}.
  \end{align*}
  For $k=3$, we have
  \begin{align*}
    \eta_3 = \min\left\{\eta_2,\ \frac{1}{4L_2}\right\} = \min\left\{\frac{1}{4\hat{L}_1},\ \frac{1}{4L_2}\right\} = \frac{1}{4\hat{L}_2} \geq \frac{3}{12\hat{L}_2}.
  \end{align*}
  Now assume that for some $k \geq 3$, $\eta_k \geq \frac{3 + \alpha(k-3)}{12\hat{L}_{k-1}}$.
  Then, using the stepsize rule, we have
  \begin{align*}
    \eta_{k+1} &= \min\left\{\frac{4}{3}\eta_k, \frac{\tau_{k-1} + 1}{\tau_k}\eta_k, \frac{\tau_k}{4L_k}\right\}\\
    &\geqby{(a)}\min\left\{\frac{k+1}{k}\eta_k,\ \frac{1+\alpha(k-2)/2}{4L_k}\right\}\\
    &\geq \min\left\{\frac{k+1}{k}\cdot \frac{3+\alpha(k-3)}{12\hat{L}_{k-1}}, \frac{2+\alpha(k-2)}{8L_k}\right\}
    \geq
    \frac{3+\alpha(k-2)}{12\hat{L}_k},
  \end{align*}
  where (a) follows from~\eqref{ineq:tau-ratio}.
  Since $\alpha\in[0,1]$ and $k\geq3$, both terms in the minimum are at least $\frac{3 + \alpha(k-2)}{12\hat{L}_k}$, and hence
  \begin{align*}
    \eta_{k+1} \geq \frac{3 + \alpha(k-2)}{12\hat{L}_k},
  \end{align*}
  which completes the induction.
\end{proof}

\subsection{Proof of~\cref{theorem:strongly-convex-linear}}\label{app:proof-strongly-convex-linear}
\begin{proof}[Proof of~\cref{theorem:strongly-convex-linear}]
  We prove the claim by induction. The base case $s=0$ holds trivially.
  Suppose that the claim holds for some $s \geq 0$, \ie, $f(w_s) - f(x^*) \leq \frac{\phi_0}{2^s}$.
  The sequence $\{w_s\}$ is bounded because $\{x_k\}$ is bounded from~\cref{theorem:convex-obj-decrease}, $w_s = x_N$ for some $N$, and $P$ is compact.
  Using the local strong convexity of $f$, we have
  \begin{align*}
    \|w_s - x^*\|^2 \leq \frac{2}{\mu}(f(w_s) - f(x^*)) \leq \frac{2\phi_0}{2^s\mu},
  \end{align*}
  where $\mu$ is the local strong convexity parameter of $f$.
  Let $D_0^2 := \frac{2\phi_0}{2^s\mu}$, an upper bound on $\|w_s-x^*\|^2$. The next stage uses $\delta_k=(1-\beta)D_0^2/(4N\eta_k)$. Since the stepsize rule gives $\eta_{i+1}\leq2\eta_i$, we have
  \begin{align*}
    \Scal_N = \frac{(1-\beta)D_0^2}{4N}\sum_{i=1}^N\left(1+\frac{\eta_{i+1}}{\eta_i}\right) \leq \frac{3(1-\beta)D_0^2}{4}.
  \end{align*}
  With $x_0=y_0=z_0=w_s$, $x_N=w_{s+1}$, and the first-iteration line search, we have $\E\leq D_0^2/(2\beta)$ and $\hat L_N\leq5\gamma L/[8(1-\beta)]$. Substituting these bounds into~\eqref{ineq:convex-sublinear-bound}, we obtain
  \begin{align*}
    f(w_{s+1}) - f(x^*) &\leq\frac{15\gamma LD_0^2}{4N(N+1)}\left(\frac{2}{\beta(1-\beta)} + 3\right)\\ &= \frac{15\gamma L}{4N(N+1)}\left(\frac{2}{\beta(1-\beta)} + 3\right)\cdot\frac{2\phi_0}{2^s\mu}
    \leq \frac{\phi_0}{2^{s+1}},
  \end{align*}
  Thus, by induction, we have $f(w_s) - f(x^*) \leq \frac{\phi_0}{2^s}$ for all $s \geq 0$.

  Next, let $S:= \lceil\log_2\max\{\phi_0/\epsilon, 1\}\rceil$.
  Excluding rejected initialization trials, the number of gradient computations required to achieve $f(w_s) - f(x^*) \leq \epsilon$ is $O(NS)$, which is bounded by $O(N_0\log_2\lceil\max\{\phi_0/\epsilon,1\}\rceil)$.
  Finally, we analyze the total number of linear minimizations.
  Let $T_{i,k}$ be the number of CG iterations performed to solve~\eqref{eq:acfg-subprob} up to the accuracy $\delta_k$ at the $k$-th outer iteration of the $i$-th stage.
  From $T_k$ in~\cref{theorem:cg-complexity}, it holds that
  \begin{align*}
    T_{i,k} \leq \frac{6D^2}{\eta_k \delta_k} + 1 \leqby{(a)} \frac{6\mu D^2 2^i N}{(1-\beta)\phi_0} + 1,
  \end{align*}
  where (a) follows from the definition of $\delta_k$ in~\cref{alg:adcgs-strongly-convex}.
  Therefore, the total number of CG iterations performed in the first $S$ stages is at most
  \begin{align*}
    \sum_{s=1}^S\sum_{k=1}^N T_{s,k} &\leq \sum_{s=1}^S\sum_{k=1}^N \frac{6\mu D^2 2^s N}{(1-\beta)\phi_0} + NS\\
    &= \frac{6\mu D^2 N^2}{(1-\beta)\phi_0}\sum_{s=1}^S 2^s + NS\\
    &\leq \frac{6\mu D^2 N^2}{(1-\beta)\phi_0}2^{S+1} + NS\\
    &\leqby{(b)} \frac{24\mu D^2 N^2}{(1-\beta)\epsilon} + NS
  \end{align*}
  where (b) follows from the definition of $S$.
  This inequality is bounded by~$O(L D^2/\epsilon + N_0\log_2\lceil\max\{\phi_0/\epsilon,1\}\rceil)$.
\end{proof}

\section{Additional Experimental Results}\label{app:additional-experimental-results}
\subsection{Least Squares Regression over the Simplex}\label{app:quadratic_simplex}
We provide additional experimental results for least squares regression over the simplex, including results for the first-iteration line-search strategy in~\cref{alg:ls-first-iteration}.
In \cref{fig:least-squares-simplex-base-200-1000}, the proposed method achieves the best performance.
In particular, \adcgs:0.5 outperforms both the baseline methods and the other \adcgs variants.
In \cref{fig:least-squares-simplex-adcgs-200-1000}, the first-iteration line-search variants remain stable, whereas \adcgs with $\alpha = 1.0$ is the exception.

\subsection{$\ell_p$ Loss Regression over the $\ell_2$-Ball}\label{app:lp_loss}
We provide additional experimental results for $\ell_p$ loss regression over the $\ell_2$-ball.
In \cref{fig:cpusmall-lp-loss}, ACG attains good objective values but exhibits instability, whereas the proposed methods and AC-FGM perform consistently well.

\subsection{Logistic Regression}\label{app:logistic_regression}
We provide additional experimental results for logistic regression.
Overall, the results are favorable: \adcgs:0.1 performs best on \cref{fig:logistic-regression-a9a,fig:logistic-regression-phishing,fig:logistic-regression-w8a}.
\begin{figure*}[!t]
  \centering
  \includegraphics[width=0.96\textwidth]{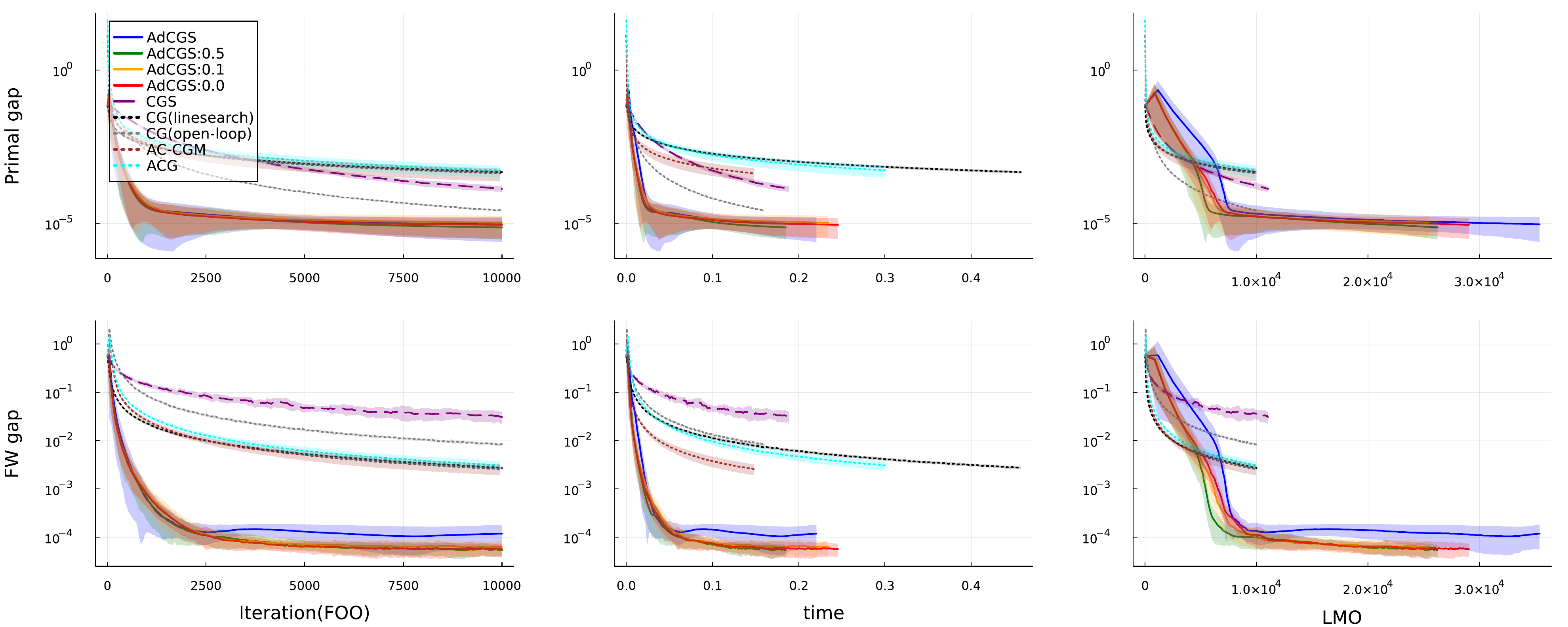}
  \caption{Least squares regression over the simplex with $(m,n) = (1000,200)$.}\label{fig:least-squares-simplex-base-200-1000}
\end{figure*}
\begin{figure*}[!t]
  \centering
  \includegraphics[width=0.96\textwidth]{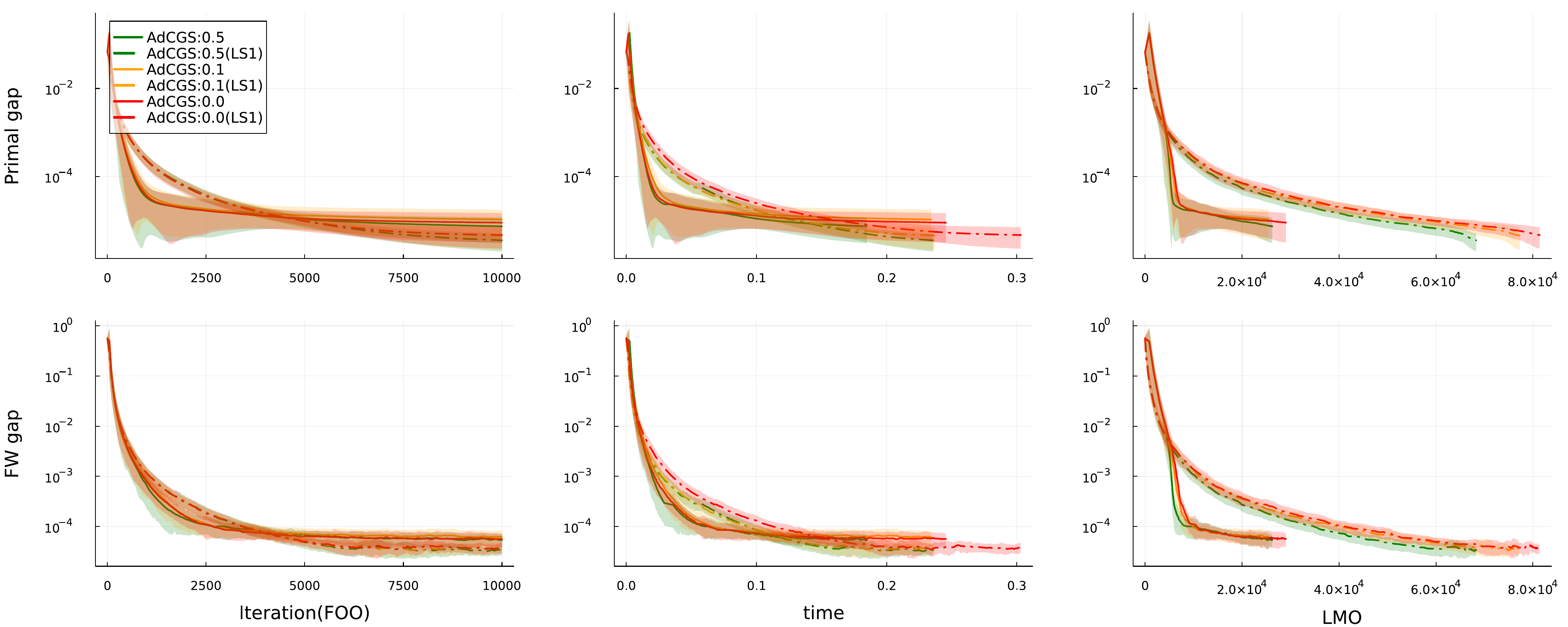}
  \caption{Least squares regression over the simplex with $(m,n) = (1000,200)$.}\label{fig:least-squares-simplex-adcgs-200-1000}
\end{figure*}
\begin{figure*}[!t]
  \centering
  \includegraphics[width=0.96\textwidth]{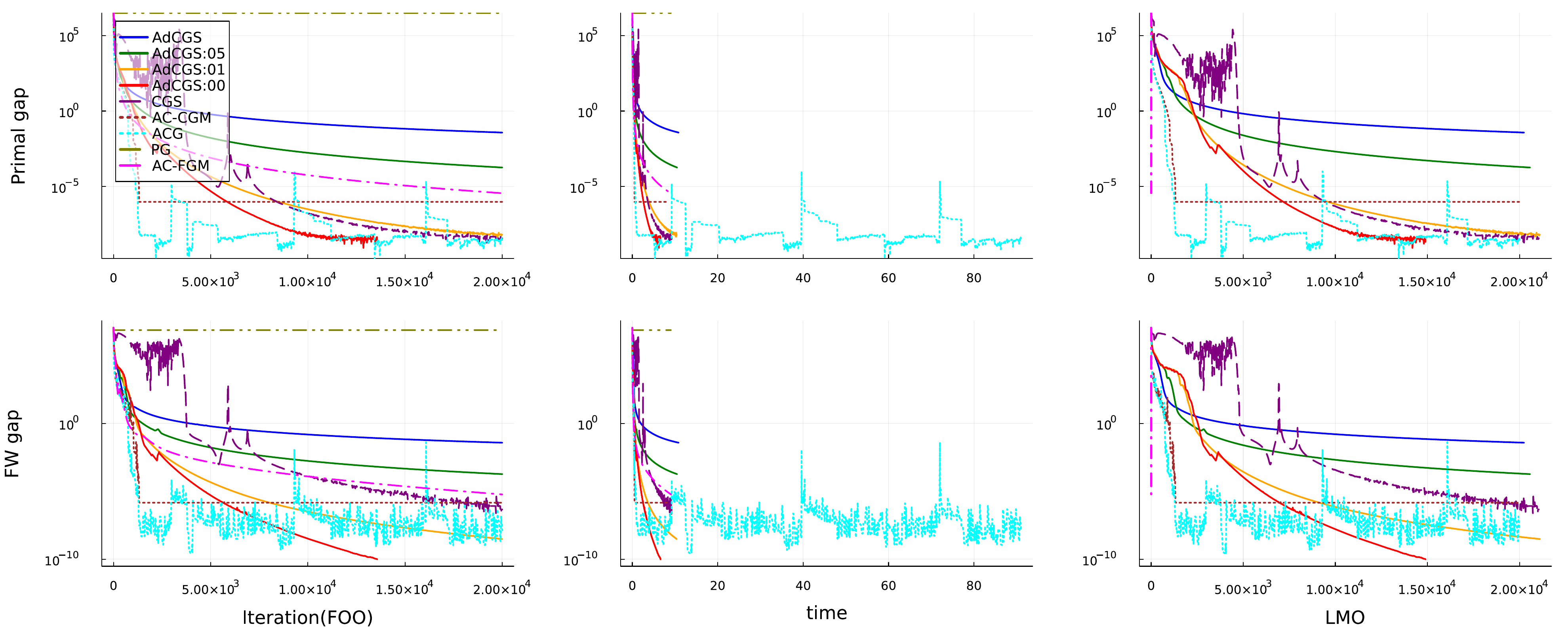}
  \caption{$\ell_p$ loss regression over the cpusmall dataset with $(m,n) = (8192,12)$.}\label{fig:cpusmall-lp-loss}
\end{figure*}
\begin{figure*}[!t]
  \centering
  \includegraphics[width=0.96\textwidth]{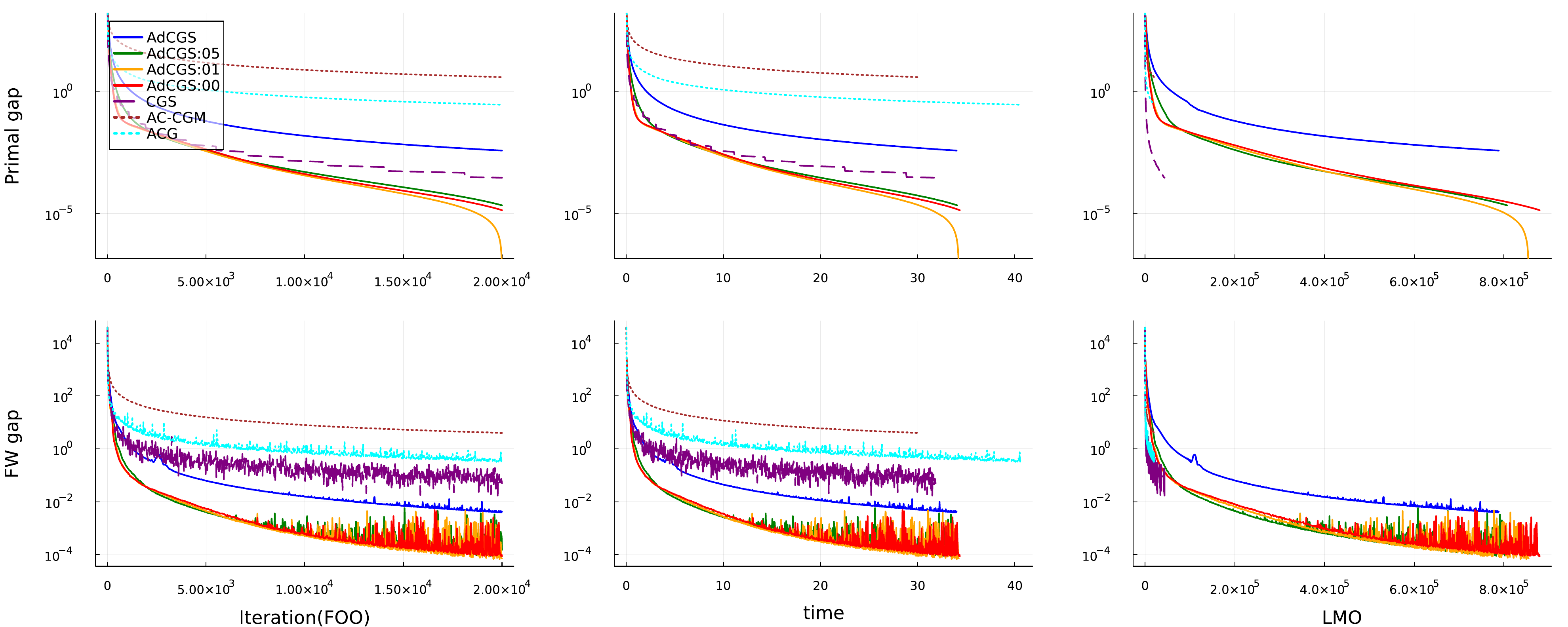}
  \caption{Logistic regression over the a9a dataset with $(m,n) = (32561,123)$.}\label{fig:logistic-regression-a9a}
\end{figure*}
\begin{figure*}[!t]
  \centering
  \includegraphics[width=0.96\textwidth]{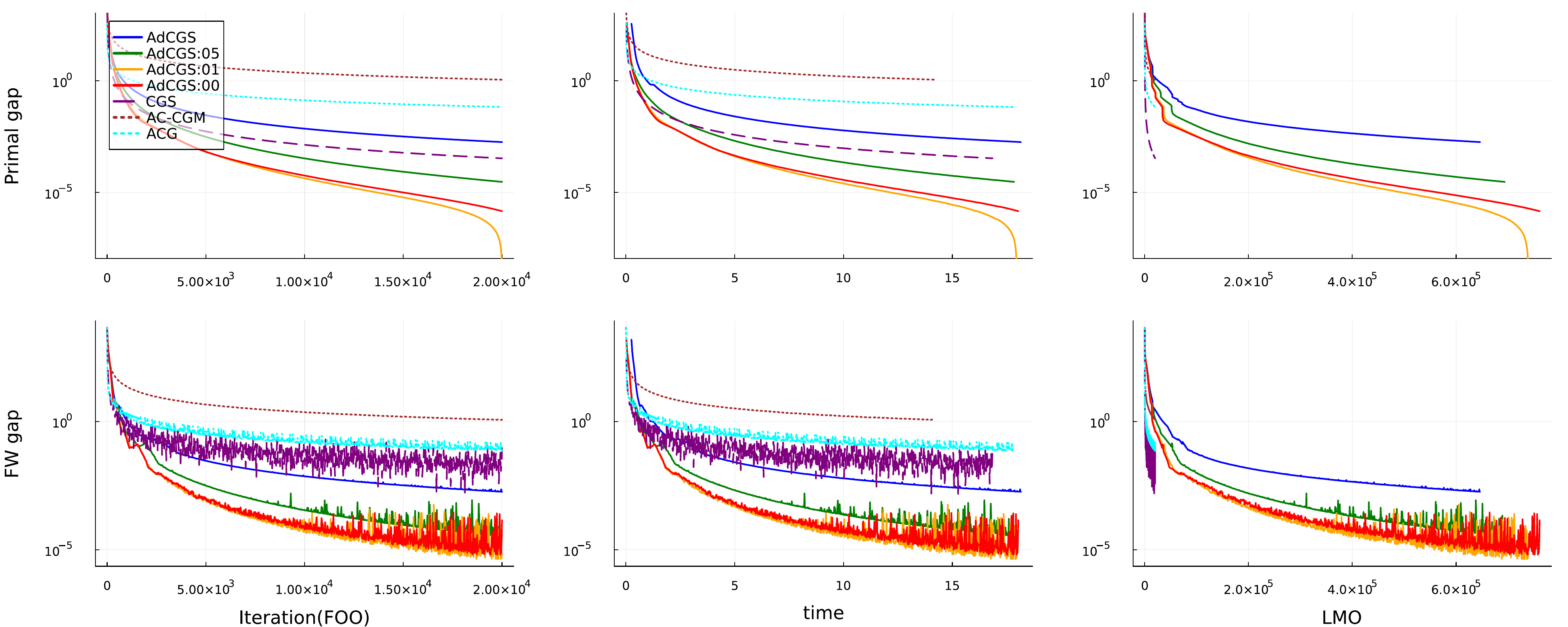}
  \caption{Logistic regression over the phishing dataset with $(m,n) = (11055,68)$.}\label{fig:logistic-regression-phishing}
\end{figure*}
\begin{figure*}[!t]
  \centering
  \includegraphics[width=0.96\textwidth]{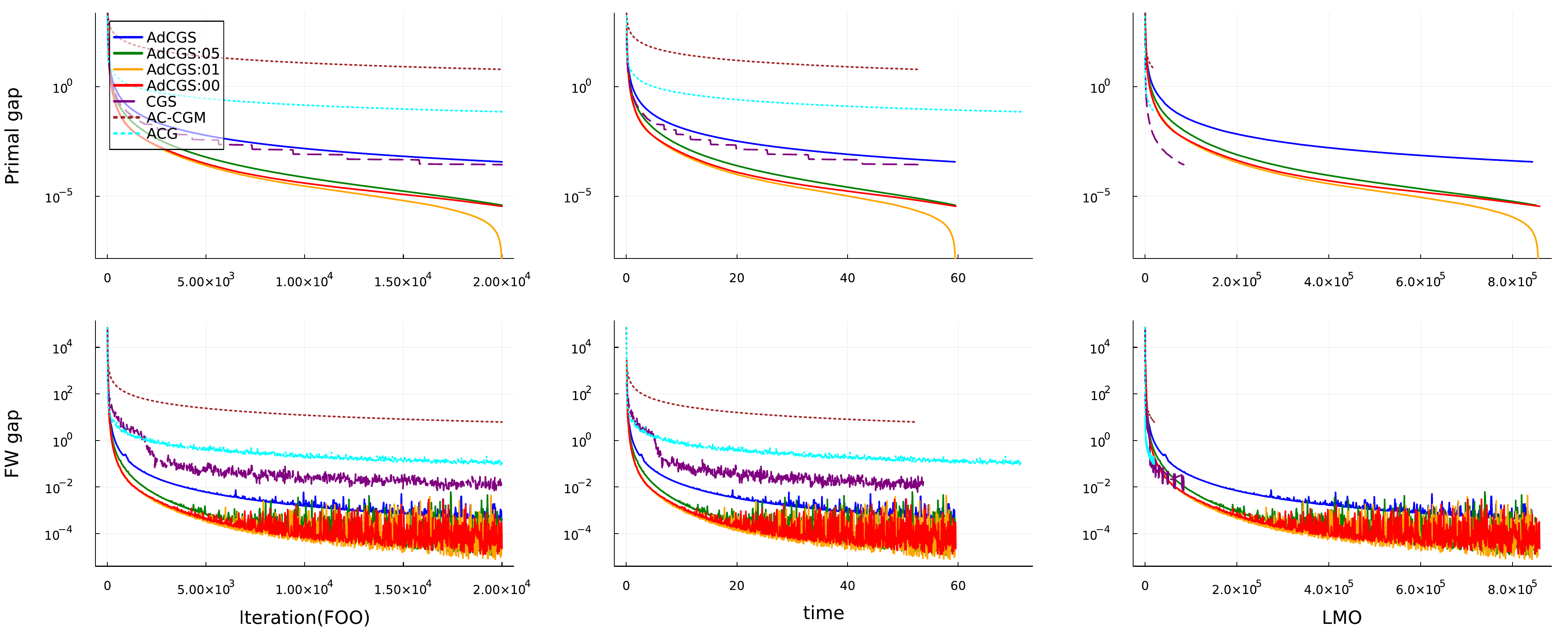}
  \caption{Logistic regression over the w8a dataset with $(m,n) = (49749,300)$.}\label{fig:logistic-regression-w8a}
\end{figure*}
\clearpage
\subsection{Poisson Regression over Positive Semidefinite Matrices with a Trace Constraint}\label{app:psd-poisson}
We consider the following Poisson negative log-likelihood problem over positive semidefinite matrices with a trace constraint:
\begin{align*}
  \min_{X\succeq0, \tr(X)=\tau}\quad
  \frac{1}{m}\sum_{i=1}^{m}(\mu_i(X)-y_i\log\mu_i(X)),
\end{align*}
where $\S^d$ denotes real symmetric matrices, $y_i\geq0$ is the observed count, $\tau>0$ is the prescribed trace, and $\mu_i(X)=q_i^\top Xq_i+b_i$ for measurement vectors $q_i\in\R^d$ and background intensities $b_i>0$.
Since $\mu_i(X)\geq b_i>0$, $f$ is convex and smooth on the feasible set, though its global smoothness bound can be conservative.
The LMO returns $\tau vv^\top$, where $v$ is a unit eigenvector for the smallest eigenvalue of the gradient matrix; Euclidean projection requires a full eigendecomposition and simplex projection of the eigenvalues.
All methods start at $X_0=(\tau/d)I$.

We set $(d,m)=(250,2000)$ and $\tau=1$.
\Cref{fig:psd-poisson-results} shows convergence under a common outer-iteration budget.
Curves and shaded bands show the mean and standard deviation over 10 instances.
The \adcgs variants achieve final primal gaps comparable to AC-FGM in less time.

\begin{figure*}[!hbp]
  \centering
  \includegraphics[width=0.96\textwidth]{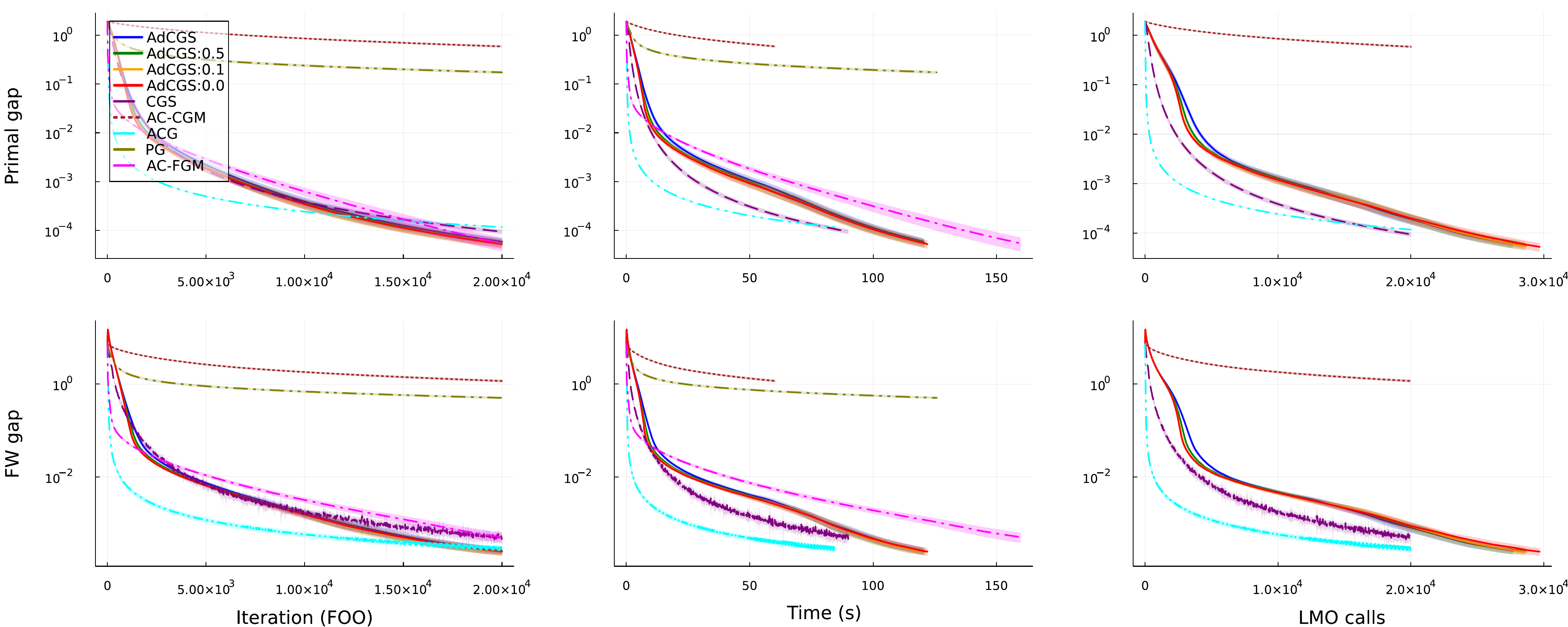}
  \caption{Poisson regression over positive semidefinite matrices with $(d,m) = (250,2000)$.}
  \label{fig:psd-poisson-results}
\end{figure*}

\clearpage
\subsection{LMO Calls per Gradient Evaluation}\label{app:lmo-gradient-ratios}
\Cref{tab:lmo-gradient-ratios-least-squares,tab:lmo-gradient-ratios-lp-loss,tab:lmo-gradient-ratios-logistic,tab:lmo-gradient-ratios-poisson} report total LMO calls divided by total gradient evaluations, measuring the average inner CG iterations per gradient evaluation.
We report each setting separately because relative oracle costs depend on the problem.

\begin{table}[!htbp]
  \centering
  \caption{LMO calls per gradient evaluation for least squares regression.}
  \label{tab:lmo-gradient-ratios-least-squares}
  \begin{tabular}{lrrrrr}
    \toprule
    Setting & \adcgs:1.0 & \adcgs:0.5 & \adcgs:0.1 & \adcgs:0.0 & CGS \\
    \midrule
    $m=1000$, $n=200$ & 3.5 & 2.6 & 2.5 & 2.9 & 1.1 \\
    $m=2500$, $n=500$ & 2.5 & 2.3 & 2.7 & 2.9 & 1.0 \\
    \bottomrule
  \end{tabular}
\end{table}

\begin{table}[!htbp]
  \centering
  \caption{LMO calls per gradient evaluation for $\ell_p$ loss regression.}
  \label{tab:lmo-gradient-ratios-lp-loss}
  \begin{tabular}{lrrrrr}
    \toprule
    Dataset & \adcgs:1.0 & \adcgs:0.5 & \adcgs:0.1 & \adcgs:0.0 & CGS \\
    \midrule
    cpusmall & 1.0 & 1.0 & 1.1 & 1.1 & 1.1 \\
    YearPredictionMSD & 1.8 & 1.8 & 2.2 & 2.4 & 1.2 \\
    \bottomrule
  \end{tabular}
\end{table}

\begin{table}[!htbp]
  \centering
  \caption{LMO calls per gradient evaluation for logistic regression.}
  \label{tab:lmo-gradient-ratios-logistic}
  \begin{tabular}{lrrrrr}
    \toprule
    Dataset & \adcgs:1.0 & \adcgs:0.5 & \adcgs:0.1 & \adcgs:0.0 & CGS \\
    \midrule
    a9a & 39 & 40 & 43 & 44 & 2.2 \\
    gisette & 50 & 50 & 50 & 50 & 10 \\
    phishing & 32 & 35 & 37 & 38 & 1.0 \\
    w8a & 42 & 43 & 43 & 43 & 4.2 \\
    \bottomrule
  \end{tabular}
\end{table}

\begin{table}[!htbp]
  \centering
  \caption{LMO calls per gradient evaluation for Poisson regression.}
  \label{tab:lmo-gradient-ratios-poisson}
  \begin{tabular}{lrrrrr}
    \toprule
    Setting & \adcgs:1.0 & \adcgs:0.5 & \adcgs:0.1 & \adcgs:0.0 & CGS \\
    \midrule
    $d=250$, $m=2000$ & 1.4 & 1.4 & 1.4 & 1.5 & 1.0 \\
    \bottomrule
  \end{tabular}
\end{table}

\end{document}